\documentclass[11pt,reqno]{amsart}
\usepackage[english]{babel}
\usepackage[utf8]{inputenc}
\usepackage[T1]{fontenc}
\usepackage{amsmath,amsfonts,amssymb,amsthm}
\usepackage{array}
\usepackage{mathdots}
\usepackage{mathtools} 
\usepackage{hyperref}
\usepackage{enumitem}
\usepackage{cleveref} 
\usepackage{graphicx}
\usepackage{xspace} 

\DeclareUnicodeCharacter{2264}{\ensuremath{\leqslant}}
\DeclareUnicodeCharacter{2265}{\ensuremath{\geqslant}}
\DeclareUnicodeCharacter{2260}{\ensuremath{\neq}}

\makeatletter 
\def\l@subsection{\@tocline{2}{0pt}{2pc}{6pc}{}}
\makeatother

\newtheorem{mainthm}{Theorem}

\theoremstyle{plain}
\newtheorem{lem}{Lemma}[section]
\newtheorem{thm}[lem]{Theorem}
\newtheorem{cor}[lem]{Corollary}
\newtheorem{prop}[lem]{Proposition}

\theoremstyle{definition}
\newtheorem{algo}[lem]{Algorithm}
\newtheorem{defi}[lem]{Definition}
\newtheorem{rem}[lem]{Remark}
\newtheorem{ex}[lem]{Example}
\newtheorem{mainconj}[mainthm]{Conjecture}

\numberwithin{figure}{section}
\numberwithin{table}{section}
\numberwithin{equation}{section}

\crefname{section}{Section}{Sections} 
\crefname{table}{Table}{Tables} 

\crefname{thm}{Theorem}{Theorems} 
\crefname{mainthm}{Theorem}{Theorems} 
\crefname{prop}{Proposition}{Propositions} 
\crefname{cor}{Corollary}{Corollaries} 
\crefname{lem}{Lemma}{Lemmas} 
\crefname{ex}{Example}{Examples} 
\crefname{conj}{Conjecture}{Conjectures} 
\crefname{algo}{Algorithm}{Algorithms} 
\crefname{defi}{Definition}{Definitions} 

\AddToHook{env/prop/begin}{\crefalias{lem}{prop}}
\AddToHook{env/thm/begin}{\crefalias{lem}{thm}}
\AddToHook{env/cor/begin}{\crefalias{lem}{cor}}
\AddToHook{env/algo/begin}{\crefalias{lem}{algo}}
\AddToHook{env/defi/begin}{\crefalias{lem}{defi}}
\AddToHook{env/rem/begin}{\crefalias{lem}{rem}}
\AddToHook{env/ex/begin}{\crefalias{lem}{ex}}
\AddToHook{env/exemples/begin}{\crefalias{lem}{exemples}}
\AddToHook{env/mainconj/begin}{\crefalias{mainthm}{mainconj}}

\setlist{topsep= 2pt plus 1pt minus 1pt,itemsep= 1pt plus 1pt minus 1pt}
\setlist[itemize]{leftmargin=\parindent}
\setlist[itemize,1]{label=$\bullet$}%
\setlist[itemize,2]{label=$\circ$}%
\setlist[enumerate,1]{leftmargin=2em,label=\arabic*), ref=\arabic*,itemjoin*=\quad}

\newcommand{\R}{\mathbb{R}}
\newcommand{\Z}{\mathbb{Z}}

\newcommand{\F}{\mathbb{F}}
\newcommand{\Q}{\mathbb{Q}}

\newcommand{\cK}{\mathcal{K}}
\newcommand{\cS}{\mathcal{S}}

\newcommand{\PSL}{\mathrm{PSL}}

\newcommand\al{\alpha}
\newcommand\be{\beta}
\newcommand\de{\delta}
\newcommand\De{\Delta}
\newcommand\eps{\varepsilon}

\newcommand\ga{\gamma}
\newcommand\Ga{\Gamma}

\newcommand\alg{{\tt NextABC}\xspace}

\makeatletter
\DeclareRobustCommand{\CF}{\DOTSB\gaussk@\slimits@}
\newcommand{\gaussk@}{\mathop{\vphantom{\sum}\mathpalette\bigcal@{K}}}

\newcommand{\bigcal@}[2]{%
  \vcenter{\m@th
    \sbox\z@{$#1\sum$}%
    \dimen@=\dimexpr\ht\z@+\dp\z@
    \hbox{\resizebox{!}{0.8\dimen@}{$\mathcal{K}$}}%
  }%
}
\newcommand{\cfp}{\,\mathrel{\cfracplus@}\,}
\newcommand{\cfracplus@}{%
  \sbox\z@{$\dfrac{1}{1}$}%
  \sbox\tw@{$+$}%
  \raisebox{\dimexpr\dp\tw@-\dp\z@\relax}{$+$}%
}
\newcommand{\cfd}{\mathord{\cfracdots@}}
\newcommand{\cfracdots@}{%
  \sbox\z@{$\dfrac{1}{1}$}%
  \sbox\tw@{$+$}%
  \raisebox{\dimexpr\dp\tw@-\dp\z@\relax}{$\cdots$}%
}
\makeatother

\let\fr=\frac
\newcommand{\cron}{\langle n\rangle_q}

\renewcommand{\mod}[1]{\ (\mathrm{mod}\ #1)}

\begin{document}

\title[$q$-deformed metallic numbers]
{Hankel continued fractions\\ and Hankel determinants\\ for $q$-deformed metallic numbers}
\date{07/02/2025}


\author{Guo-Niu Han}
\address{%
Guo-Niu Han,
Université de Strasbourg, CNRS, IRMA, UMR~7501, Strasbourg, France}
\email{guoniu.han@unistra.fr}

\author{Emmanuel Pedon}
\address{%
Emmanuel Pedon,
Universit\'e de Reims Champagne-Ardenne,
CNRS, LMR, UMR~9008, Reims, France} 
\email{emmanuel.pedon@univ-reims.fr}

\begin{abstract}
Fix $n$ a positive integer. Take the $n$-th \emph{metallic number} 
$$\phi_n=\frac{n+\sqrt{n^2+4}}{2}$$
(e.g. $\phi_1$ is the golden number) and let $\Phi_n(q)$ be its $q$-deformation in the sense of S.~Morier-Genoud \& V.~Ovsienko. This is an algebraic continued fraction which admits an expansion into a Taylor series $\Phi_n(q) =\sum_{i=0}^{+\infty} f_i q^i$ around $q=0$, with integral coefficients. Consider the sequences of shifted
Hankel determinants of $\Phi_n$:
$$\De_j^{(\ell)}:=\det(f_{a+b+\ell})_{a,b=0}^{j-1}, \quad \ell,j≥0.$$
By using the notion of Hankel continued fraction introduced by the first author  in \cite{Han16} we determine explicitly the first $n+2$ sequences $\Delta_j^{(0)},\Delta_j^{(1)},\dots,\Delta_j^{(n+1)}$ and show that they satisfy the following properties:
\begin{enumerate}
\item They are periodic and consist of~$-1,0,1$ only.
\item They satisfy a three-term Gale-Robinson recurrence, i.e. they form discrete integrable dynamical systems.
\item They are all completely determined by the first sequence $\Delta_j^{(0)}$.
\end{enumerate}
This article thus validates a conjecture formulated by V.~Ovsienko and the second author in \cite{OP25} and establishes new connections between $q$-deformations of real numbers and  sequences of Catalan or Motzkin numbers.
\end{abstract}

\maketitle

\thispagestyle{empty}

\tableofcontents

\section{Introduction}

Given a power series $F(q)=\sum_{i=0}^{+\infty}f_iq^i$
or simply a sequence of numbers $F=(f_i)_{i\in\Z_{≥0}}$, by the \emph{Hankel determinants of $F$} we mean the doubly indexed family 
\begin{equation}
\label{HankDet}
\De^{(\ell)}_j(F):=\left|
\begin{matrix}
{f}_\ell & {f}_{\ell+1} & \cdots & {f}_{\ell+j-1} \\[4pt]
{f}_{\ell+1} &{f}_{\ell+2} & \cdots & {f}_{\ell+j} \\[4pt]
\vdots & \vdots && \vdots \\
 {f}_{\ell+j-1} & {f}_{\ell+j} & \cdots & {f}_{\ell+2j-2}
\end{matrix}
\right|
\end{equation}
where $\ell,j\in\Z_{≥0}$, with the convention $\De^{(\ell)}_0(F):=1$. The determinant $\De^{(\ell)}_j(F)$ is thus of size $j\times j$, and we call the number $\ell$ the \emph{shift} of the determinant, since it is clear that
\begin{equation*}
\De^{(\ell)}_j\bigl((f_i)_{i≥0})\bigr)=\De_j^{(0)}\bigl((f_{i+\ell})_{i≥0})\bigr).
\end{equation*}
Similarly, if  $F(q)=\sum_{i=0}^{\infty}f_iq^i$ is a power series, we have
\begin{equation}\label{defFl}
\De^{(\ell)}_j(F)=\De_j^{(0)}(F^{(\ell)})
\end{equation}
where
\begin{equation*}
F^{(\ell)}(q):=
\frac 1{q^\ell}\left(F(q)-\sum_{i=0}^{\ell-1}f_iq^i\right)=\frac 1{q^\ell}\left(\sum_{i=\ell}^{+\infty}f_iq^i\right).
\end{equation*}
When $\ell=0$, we will often use the shorter notation $\De_j(F)$ instead of $\De^{(0)}_j(F)$ and we speak of \emph{ordinary Hankel determinants} of $F$. On the other hand we will write simply $\De^{(\ell)}(F)$ to denote the entire sequence $(\De^{(\ell)}_j(F))_{j≥0}$.

First appeared in 1861 \cite{Hankel61}, Hankel matrices and determinants (and later on,  Hankel and Toeplitz operators) are this kind of objects that arise naturally in various fields of mathematics and often interconnect them: combinatorics \cite{Radoux79,Aigner99,Tamm01,CK11}, orthogonal polynomials \cite{Viennot83,Junod03,DIK11,BP17},  number theory \cite{Wall,APWW98,Bugeaud11}, probability theory \cite{Schmudgen,BDJ06},  complex analysis \cite{Henrici1,Pommerenke66,KLS18}, functional and applied analysis \cite{Peller,Sauer}… (among many other valuable references).  

In this article we would like to present the remarkable properties of  Hankel determinants of a family of power series $\Phi_n$ ($n\in\Z_{≥1}$) that are defined as ``$q$-deformations'', in the sense of S.~Morier-Genoud and V.~Ovsienko (to be specified a bit later), of the so-called \emph{metallic numbers}, or \emph{metallic ratios}, or \emph{metallic means}. Recall that these  numbers are the quadratic irrationals

\begin{equation*}
\phi_n:=\frac{n+\sqrt{n^2+4}}{2} = n+\cfrac{1}{n+\cfrac{1}{n+\cfrac 1{\ddots}}}, \quad n≥1,
\end{equation*} 
with most famous representatives:
\begin{align*}
	\phi_1&=\frac{1+\sqrt{5}}{2}\qquad\text{(golden ratio)},\\
	\phi_2&=1+\sqrt{2}\qquad\text{(silver ratio)},\\
	\phi_3&=\frac{3+\sqrt{13}}{2}\qquad\text{(bronze ratio)}.
\end{align*}
Surprisingly, the properties of their Hankel determinants, together with some particular continued fraction expansions, will make our functions $\Phi_n$ resemble  to the generating series of Catalan or Motzkin numbers, which are ubiquitous in mathematics. This connection was first observed in cases $n=1$ and $n=2$ by V.~Ovsienko and the second author \cite{OP25}. The present work is thus a continuation and generalization of the results therein. 

Let us elaborate. First, we recall that the \emph{Catalan numbers} are defined by the formula\footnote{We avoid the traditional notation $C_n$ because of possible confusion with some other object introduced later.}
\begin{equation}\label{Catalan}
\gamma_n:=\fr{1}{n+1}\binom{2n}{n}=\fr{(2n)!}{n!(n+1)!}, \quad n≥0,
\end{equation}
and that first terms are
\begin{equation*}
1, 1, 2, 5, 14, 42, 132, 429, 1430, 4862, 16796, 58786, 208012, 742900, 2674440,\,\ldots
\end{equation*} 
see sequence A000108 in the On-line Encyclopedia of Integer Sequences (OEIS) \cite{oeis}. All shifted Hankel determinants of the generating series $\Ga(q):=\sum_{n=0}^{+\infty}\ga_n q^n$  are known, and expressed by the following simple formulas (see e.g. \cite{Aigner99,DCV86,Krattenthaler05}):
\begin{align}
&\De_j(\Ga)=1,\,1,\,1,\,1,\,\ldots,
&
\De^{(1)}_j(\Ga)&=1,\,1,\,1,\,1,\,\ldots,
\label{Hank01Cat}\\
&\De^{(2)}_j(\Ga)=1,\,2,\,3,\,4,\,\ldots,
&
\De^{(\ell)}_j(\Ga)&=\prod_{1≤a≤b≤\ell-1}\fr{2j+a+b}{a+b},\quad \ell≥2.\label{Hank2Cat}
\end{align}
Another classical example is that of the \emph{Motzkin numbers}
\begin{equation*}
M_n:=\sum_{k=0}^{\lfloor n/2\rfloor}\binom{n}{2k}\ga_k, \quad n≥0,
\end{equation*}
that form  sequence A001006 in OEIS:
\begin{equation*}
1, 1, 2, 4, 9, 21, 51, 127, 323, 835, 2188, 5798, 15511, 41835, 113634,\,\ldots
\end{equation*} 
In this case, the first four sequences of Hankel determinants of the generating series $M(q):=\sum_{n=0}^{+\infty}M_n q^n$ are given by
\begin{align}
\De_j(M)&=1,\,1,\,1,\,1,\,\ldots,\label{Hank0Motzkin}\\
\De^{(1)}_j(M)&=1,\,1,\,0,-1,-1,\,0,1,\,1,\,0,-1,-1,\,0,\,\ldots,\label{Hank1Motzkin}\\
\De^{(2)}_j(M) &= 1,\,2,\,2,\,3,\,4,\,4,\,5,\,6,\,6,\,7,\,8,\,8,\,\ldots,\label{Hank2Motzkin}\\
\De^{(3)}_j(M) &=1, 4, 3, -6, -16, -10, 15, 36, 21, -28, -64, -36, 45,\,\ldots,\notag
\end{align}
see e.g. \cite{Aigner98,BP17}.
It is particularly interesting to note that the  sequence \eqref{Hank0Motzkin} of ordinary Hankel determinants is identically~$1$ as for Catalan numbers, while the  first shifted sequence \eqref{Hank1Motzkin} \emph{is $3$-antiperiodic (thus $6$-periodic) and consists of $-1,0$, and $1$ only} (see entry A010892 in OEIS). The sequences $\De^{(\ell)}(\Phi_n)$ we are going to consider in the present article  share this special property, as will be explained below. For later comparison, let us mention also that the shifted Hankel sequence $\De^{(1)}_n(M)$ satisfies
the recurrence relation
\begin{equation}\label{SomosMotzkin}
\De_{n+2}\De_n=\De_{n+1}^2-1.
\end{equation}

Let us come to the  presentation of our results. Throughout the paper we will use the following classical notation for continued fractions:
\begin{equation*}
\fr{\al_0}{\be_0}\cfp\fr{\al_1}{\be_1}\cfp\fr{\al_2}{\be_2}\cfp\cfd\cfp\fr{\al_{p-1}}{\be_{p-1}}\cfp\fr{\al_p}{\be_p}:=
\cfrac{\al_0}{\be_0+\cfrac{\al_1}{\be_1+\cfrac{\al_2}{\ddots+\cfrac{\al_{p-1}}{\be_{p-1}+\cfrac{\al_p}{\be_p}}}}}
\end{equation*}
and, to shorten, Gauss' notation
\begin{equation*}
\CF_{j=0}^{p}\fr{\al_j}{\be_j}:=\fr{\al_0}{\be_0}\cfp\fr{\al_1}{\be_1}\cfp\cfd\cfp\fr{\al_p}{\be_p}.
\end{equation*}
Both notation are extended  to infinite continued fractions, and periodic data inside a continued fraction will be indicated between  parentheses and with a star superscript $*$. For instance, we remind the  noticeable expression of the metallic number $\phi_n$ as a $1$-periodic regular continued fraction: 
\begin{equation*}
\phi_n=n+\left(\fr 1{n}\cfp\right)^*=n+\fr{1}{n}\cfp\fr{1}{n}\cfp\fr{1}{n}\cfp\cfd
\end{equation*}
From now on, for any positive integer $n$ we denote by $\Phi_n$ the \emph{$q$-deformation of $\phi_n$} in the sense of S.~Morier-Genoud and V.~Ovsienko \cite{MGO22}. By definition, it is the 2-periodic continued fraction
\begin{equation}\label{defPhin}
\Phi_n(q)=[n]_q+\left(\fr{q^n}{[n]_{q^{-1}}}\cfp\fr{q^{-n}}{[n]_q}\cfp\right)^*
\end{equation}
where $[n]_q$ is the polynomial which stands for the classical Euler-Gauss $q$-deformation of the integer $n$:
\begin{equation*}
[n]_q:=\frac{1-q^n}{1-q}=1+q+q^2+\cdots+q^{n-1}.
\end{equation*}
The functions $\Phi_n$ ($n≥1$) will  be refered to as  \emph{the $q$-deformed metallic numbers}, or simply \emph{the $q$-metallic numbers}. A motivation for definition \eqref{defPhin}, and more generally for the $q$-deformation of real numbers that was discovered by S.~Morier-Genoud and V.~Ovsienko a few years ago, will be given in \cref{sec-qreals}, together with some of the key results of the theory. 

As an example, let us look at the $q$-golden number $\Phi_1(q)$, i.e. the $q$-deformation of the golden ratio $\phi_1=\fr{1+\sqrt{5}}2$, whose definition \eqref{defPhin} can easily be  rewritten as
\begin{equation*}
\Phi_1(q)=1+\left(\fr{q^2}{q}\cfp\fr{1}{1}\cfp\right)^*.
\end{equation*}
In the articles \cite{MGO22} and \cite{OP25} were discovered several striking properties of $\Phi_1$   which led to a comparison with the generating series $\Ga$ and $M$ of the Catalan and Motzkin numbers. Let us indicate three such connections:
\begin{enumerate}
\item $\Phi_1(q)$ has various nice continued fraction expansions, such as
\begin{equation}\label{CfracPhi1}
\Phi_1(q)
=\fr{1}{1}\cfp\left(\fr{-q^2}{1}\cfp\fr{q}{1}\cfp\right)^*=
1+\fr{q^2}{1}\cfp\left(\fr {q}{1}\cfp\fr {q}{1}\cfp\fr {q^3}{1}\cfp\right)^*
\end{equation}   
which belong to the particular class of $C$-fractions (see \cite{OP25}, where more expansions of $\Phi_1(q)$, e.g. as $J$-fractions,  can also be found). Compare with the following well-known Catalan and Motzkin $C$-fraction expansions:
\begin{equation*}
\Ga(q)=\fr{1}{1}\cfp\left(\fr{-q}{1}\cfp\right)^*,
\qquad 
M(q)=\fr{1}{1}\cfp\left(\fr {-q}{1}\cfp\fr {-q}{1}\cfp\fr {-q^2}{1}\cfp\right)^*.
\end{equation*}
\item The $q$-golden number admits the following Taylor series expansion about $q=0$:
\begin{equation} \label{SEPhi1}
\begin{split}
\Phi_1(q)&=
1 + q^2 - q^3 + 2 q^4 - 4 q^5 + 8 q^6 - 17 q^7 + 37 q^8 - 82 q^9\\
&\quad +\, 185 q^{10}- 423 q^{11} + 978 q^{12}-2283q^{13}+ 5373q^{14}-12735q^{15} + \cdots\\
\end{split}
\end{equation} 
see~\cite{MGO22}. The coefficients here are quite close to sequence A004148 of OEIS, whose terms are called ``generalized Catalan numbers''\footnote{Note however that there are many other sequences known under this name.}; the differences are the lack of linear term in the above formula and the alternating signs. Sequence A004148 has many interesting combinatorial interpretations, such as enumeration of peakless Motzkin paths, secondary structures of RNA molecules and diagonal sums of the Narayana triangle, see entry A004148 of \cite{oeis}. 
\item Hankel determinants of $\Phi_1$ satisfy the following properties (see~\cite{OP25}): the first four sequences of Hankel determinants $\De^{(\ell)}_n(\Phi_1)$, for $\ell=0,1,2,3$, are $4$-antiperiodic
 (thus $8$-periodic): 
\begin{equation}\label{HankelGold1}
\De^{(\ell)}_{n+4}(\Phi_1)=-\De^{(\ell)}_n(\Phi_1),\quad
\ell=0,1,2,3,
\end{equation} 
and they consist of $0,1$, and $-1$ only.
First terms are:
\begin{equation}
\label{HankelGold2}
\begin{array}{rclrrrrrrrr}
\De_j(\Phi_1) &=& 1,&\,1,&\,1,&\;\;0,&-1,&-1,&-1,&\,0,&\ldots,\\
\De^{(1)}_j(\Phi_1)&=& 1,&\,0,&-1,&\;1,&-1,&\,0,&\,1,&-1,&\ldots,\\
\De^{(2)}_j(\Phi_1)&=& 1,&\,1,&\,1,&\,0,&-1,&-1,&-1,&\,0,&\ldots,\\
\De^{(3)}_j(\Phi_1)&=& 1,&-1,&\,0,&\,0,&-1,&\,1,&\,0,&\,0,&\ldots.
\end{array}
\end{equation}
These sequences compare with the Motzkin first shifted Hankel sequence \eqref{Hank1Motzkin}. To some extent, a sequence identically equal to $1$ such as \eqref{Hank01Cat} or \eqref{Hank0Motzkin} can be thought of as a particular case of a periodic sequence with values in the set $\{-1,0,1\}$, thus we  may interpret \eqref{HankelGold2} as another  connection between the $q$-golden number $\Phi_1$ and the Catalan and Motzkin series $\Ga$ and $M$.  Also, one has
\begin{equation*}
\De^{(4)}_n(\Phi_1)=1,\,2,\,0,-2,-3,-4,\,0,\,4,\,5,\,6,\,0,-6,-7,-8,\,0,\,8,\ldots,
\end{equation*}
which reminds  \eqref{Hank2Cat} and \eqref{Hank2Motzkin}. Lastly, the first three Hankel determinants sequences $\De_n(\Phi_1)$, $\De_{n}^{(1)}(\Phi_1)$ and 
$\De_n^{(2)}(\Phi_1)$ all satisfy  the Somos-4 relation 
\begin{equation}
\label{SomosPhi1}
\De_{n+4}\De_n = \De_{n+3}\De_{n+1} - \De_{n+2}^2,
\end{equation} 
which is close to \eqref{SomosMotzkin}. Recall that, for $k≥4$, a \emph{Somos-$k$ sequence} is a solution $(S_j)$ of a quadratic recurrence of the form
\begin{equation}\label{Somos}
S_j S_{j+k}= \sum_{i=1}^{\lfloor{k/2}\rfloor}\alpha_i S_{j+i}S_{j+k-i}
\end{equation} 
for arbitrary integers~$\alpha_i$ and initial data $S_0,\ldots,S_{k-1}$; see more on that subject after the statement of \cref{MainGR} below.
\end{enumerate} 

Let us say a word about the proofs of \eqref{HankelGold1}, \eqref{HankelGold2} and \eqref{SomosPhi1}.
Usual techniques to compute Hankel determinants of classical sequences (such as Catalan or Motzkin) rely mostly on continued fraction expansions, orthogonal polynomials theory and/or combinatorial models; see e.g.~\cite{Heilermann1846,Krattenthaler99,Krattenthaler05,Aigner01,Barry10,Elouafi15,BP17}…
In particular, one may think that continued fraction expansions like \eqref{CfracPhi1} can help to compute the Hankel sequences, since there exists a well-known Hankel determinant formula for \emph{regular} $C$-fractions (also known as \emph{Stieltjes continued fractions}) and for  \emph{regular} $J$-fractions (also known as \emph{Jacobi continued fractions}); see~\cite{Krattenthaler05}~§5.4. Unfortunately, it is not possible to expand $\Phi_1$ (nor any $\Phi_n$, actually) as such special continued fractions.
But, in \cite{OP25} as well as in the present work,  Hankel determinant formulas are obtained by using the notion of \emph{Hankel continued fractions} introduced by the first author  in \cite{Han16}. These form a  family of continued fractions which  includes both  classical {regular $J$-fractions} and {regular $C$-fractions} as subcases, and still conveniently comes with an explicit Hankel determinant formula; see~\cref{secH-frac} for  details.	

For instance, it was proved in \cite{OP25} (Lemmas~3.6 and 3.9) that the $H$-fraction of the $q$-golden number $\Phi_1$ is the $3$-periodic continued fraction
\begin{equation}\label{HFGold}
\Phi_1(q)=\frac{1}{1}\cfp\left(\fr{-q^2}{1+q}\cfp \fr{q^3}{1+q-q^2}\cfp \fr{q^3}{1+q}\cfp\right)^*
\end{equation}
while the $H$-fraction expansion of the $q$-silver number $\Phi_2$ is $8$-periodic and reads:
\begin{equation}\label{HFSilver}
\begin{split}
\Phi_2(q)=\frac{1}{1-q}\cfp
&\left(\fr{q^2}{1+q}\cfp \fr{q^2}{1}\cfp \fr{q^2}{1}
\cfp\fr{-q^{3}}{1+2q^2}\right.\cfp\\
&\qquad\left.\fr{q^{5}}{1+2q^2-q^3}\cfp\fr{q^{5}}{1+2q^2}\cfp\fr{-q^{3}}{1}
\cfp\fr{q^2}{1}\cfp\right)^*.
\end{split}
\end{equation} 
It turns out that these two formulas (and their analogues for the shifts $\Phi_1^{(\ell)}$ and $\Phi_2^{(\ell)}$), together with the Hankel determinant formula for $H$-fractions, yielded all the nice properties of the determinants $\De_j^{(\ell)}(\Phi_1)$ and $\De_j^{(\ell)}(\Phi_2)$ that were highlighted in \cite{OP25}. 

In the present paper we will follow the same path and begin our work by proving the generalization of \eqref{HFSilver} to any integer $n≥2$, namely:

\begin{mainthm}[{$H$-fraction expansion for $\Phi_n$}] \label{MainHF}
Suppose $n≥2$ and define the polynomial
\begin{align}
\langle n\rangle_q
&:=q[n]_q+(1+q^n)(1-q)\label{defnq}\\
&=
\begin{cases}
1+q^2+q^3+\cdots+q^{n-1}+2q^n-q^{n+1}, &\text{ if } n≥ 3,\\
1+2q^2-q^3,&\text{ if } n=2.
\end{cases}
\notag
\end{align}
The $H$-fraction expansion of the $q$-deformed metallic number $\Phi_n(q)$ 
	defined in \cref{defPhin}
	is $(6n-4)$-periodic with offset $1$ and has the following form:
\begin{equation}\label{Hfrac-Phin}
\Phi_n(q)=\frac{1}{1-q}\cfp\left(U_n(q)\cfp V_n(q)\cfp W_n(q)\cfp\right)^*,
\end{equation}
where
\begin{align*}
U_n(q)&=\CF_{i=0}^{n-3}\left(\fr{q^{n-i}}{[n-i]_q}\cfp \fr{q^n}{[i+2]_q-q}\cfp \fr{q^{i+2}}{1-q}\right)\cfp\fr{q^2}{[2]_q}\cfp \fr{q^n}{[n]_q-q}\cfp \fr{q^n}{1},
\\
V_n(q)&=\fr{-q^{n+1}}{\cron+q^{n+1}}\cfp\fr{q^{2n+1}}{\cron}\cfp
\fr{q^{2n+1}}{\cron+q^{n+1}}\cfp\fr{-q^{n+1}}{1},
\\
W_n(q)&=\CF_{i=0}^{n-3}\left(\fr{q^{n-i}}{[n-i]_q-q}\cfp \fr{q^n}{[i+2]_q}\cfp \fr{q^{i+2}}{1-q}\right)
\cfp\fr{q^2}{1}.
\end{align*}
\end{mainthm}

When $n=2$, we see that \eqref{Hfrac-Phin} coincides as expected with \eqref{HFSilver}, provided we replace all $\cK$ operators by $0$. On the contrary, when $n=1$ the $H$-fraction expansion \eqref{HFGold} of the $q$-golden number $\Phi_1$ is $3$-periodic and does not fit into the above general formula. 

Note also that an analogue of \cref{MainHF} exists for all shifts $\Phi_n^{(\ell)}$ with $\ell=1,\ldots,n+1$ (see definition \eqref{defFl}), but will not be stated here for simplicity; see \cref{HFPhishifts}.

We continue the exposition of our main results and deal now with the Hankel determinants of $\Phi_n$. First of all, as for the $q$-golden number \eqref{SEPhi1}, all $q$-metallic numbers $\Phi_n$ expand  as Taylor series about zero: actually this is a general property of $q$-deformations of  \emph{positive} real numbers, see \cite{MGO22}. In particular, all their (ordinary or shifted) Hankel determinant sequences are well-defined.
Thus, from now on we can fix a positive integer $n$ and set 
$$\De_j^{(\ell)}:=\De_j^{(\ell)}(\Phi_n)$$
for short.

We begin with the generalization of \eqref{HankelGold1}.  

\begin{mainthm}[{Values and Periodicity}]\label{MainValPer}
Sequences  $\Delta^{(\ell)}$ with $\ell\in\{0,1,\ldots,n+1\}$ consist of~$-1,0,1$ only. Moreover they are $2n(n+1)$-periodic when $n$ is even and $2n(n+1)$-antiperiodic (hence $4n(n+1)$-periodic) when $n$ is odd:
\begin{equation*}
\De_{j+2n(n+1)}^{(\ell)}=(-1)^n\De_{j}^{(\ell)}\quad\text{for all } j≥0.
\end{equation*}
\end{mainthm}

This result reflects the periodicity of the $H$-fraction expansion \eqref{Hfrac-Phin} but also reveals some symmetric patterns inside it; see~\cref{section-symmetries}. 
Explicit formulas for the Hankel determinants $\De_j^{(\ell)}$ will be given in \cref{thm-Hankel0} for the case $\ell=0$, together with a few examples for small values of $n$; cases $\ell=1,\ldots,n+1$ will then follow from \cref{MainContiguity} below.

Now we describe the dynamics of the Hankel sequences:

\begin{mainthm}[{Gale-Robinson recurrence}] \label{MainGR}
Sequences  $\Delta^{(\ell)}$ with $\ell\in\{0,1,\ldots,n+1\}$ all satisfy the same following three-term Gale-Robinson recurrence:
\begin{equation*}
\De^{(\ell)}_{j+2n+2}\,\De^{(\ell)}_j = \De^{(\ell)}_{j+2n+1}\,\De^{(\ell)}_{j+1} - \bigl(\De^{(\ell)}_{j+n+1}\bigr)^2
\quad\text{for all } j≥0.
\end{equation*}
\end{mainthm}

Recall that a \emph{three-term Gale-Robinson sequence} is a sequence $(S_j)$ defined by a recurrence of the form
\begin{equation}\label{GR3seq}
S_j S_{j+k}=\alpha S_{j+a} S_{j+k-a}+\beta S_{j+b} S_{j+k-b},
\end{equation}
where $k≥4$, $1≤a<b≤k/2$ and $\alpha,\beta$ are some integer constants. Similarly, a \emph{four-term Gale-Robinson sequence} is defined by allowing a third term $\ga S_{j+c} S_{j+k-c}$ in the right-hand side of \eqref{GR3seq}; see  \cite{Gale91}. One can see at once that Somos-4 sequences and Somos-6 sequences (definition \eqref{Somos}) are special instances of three-term Gale-Robinson sequences and four-term Gale-Robinson sequences, respectively.  Two such examples are thus provided by the Hankel determinants $\De^{(\ell)}(\Phi_n)$ ($\ell=0,1,\ldots,n+1$), for $n=1$ (remind \eqref{SomosPhi1}) and $n=2$ (see \cite{OP25} §1.4).

While Somos sequences arose from elliptic function theory and can be studied from this point of view (see e.g. \cite{Hone07,HS08}), some of their properties are best understood with S.~Fomin \& A.~Zelevinsky's theory of cluster algebras, and in particular the \emph{Laurent phenomenon} \cite{FZ02} which implies that all (three or four-term) Gale-Robinson sequences with positive integer coefficients $\al,\be,\ga$ and initial data $S_0=S_1=\cdots=S_{k-1}=1$ have the property of \emph{integrality}, i.e. all terms of such a sequence are integers. This approach gave a unifying proof of the integrality of Somos-$k$ sequences for $k=4,5,6,7$, which was already known to several authors.

Another main property of Somos-Gale-Robinson sequences is that they exhibit solutions to discrete dynamical systems which are integrable in the sense of Liouville-Arnold \cite{FH14,HLK20}. In this way,  Hankel determinants of $q$-deformed metallic numbers provide examples of \emph{periodic} $\{-1,0,1\}$-solutions of the corresponding discrete integrable systems; see the detailed example of $\Phi_1(q)$ in \cite{OP25}, §1.4.
Lastly, let us add that many examples of Somos sequences are produced by Hankel determinants, see \cite{CHX15,Hone21} and the references therein.

Let us come back to the presentation of our results. As we have seen, \cref{MainGR} points out special relations existing between consecutive Hankel determinants inside the same $\ell$-shifted sequence. Incidentally,  we have also discovered an explicit formula connecting the members of two consecutive shifted sequences, which implies in particular that all shifted sequences are completely determined by the ordinary Hankel determinant sequence $\De^{(0)}=\De(\Phi_n)$:

\begin{mainthm}[{Contiguity relations}]\label{MainContiguity}
Pairs of consecutive shifted sequences $(\Delta^{(\ell+1)},\Delta^{(\ell)})$ with $\ell\in\{0,1,\ldots,n\}$ are interconnected by the following formula:
\begin{equation}\label{contiguity}
\Delta^{(\ell+1)}_{j}=(-1)^{j+\frac{n(n+2\ell-1)}{2}}\Delta^{(\ell)}_{j+n+1}
\quad\text{for all } j≥0.
\end{equation}
\end{mainthm}

Note that Theorems~\ref{MainValPer}, \ref{MainGR} and \ref{MainContiguity} were  proved by  V.~Ovsienko and the second author in \cite{OP25} for $n=1$ and $n=2$ (see Theorem~1.6 and Theorem~1.8 therein), but only conjectured in the general case, excepting relation \eqref{contiguity} which was proved  when $\ell=n$ for any $n$ (see Proposition~4.1 therein). 

The reader has certainly noticed that all our statements  deal with shifted sequences $\De^{(\ell)}$ for a parameter $\ell$ such that $0≤\ell≤n+1$. As far as concerns the case $\ell≥n+2$, computer experimentation incite us to formulate the following conjecture, that we are not able to prove so far, except for part (1) when $n=1$, see Theorem~1.6 in \cite{OP25}.

\begin{mainconj} \emph{The case $\ell≥n+2$.}\label{Conjn+2}
\begin{enumerate}
\item The sequence  $\Delta_j^{(n+2)}$ is $2n(n+1)$-periodic when $n$ is even and $2n(n+1)$-antiperiodic (hence $4n(n+1)$-periodic) when $n$ is odd, and consists of~$-2,-1,0,1,2$ only (except if $n=1$: it can only take the values $-1,0,1$).
\item All  sequences  $\Delta_j^{(n+2)}$ with $\ell≥n+3$ are unbounded.
\end{enumerate}
\end{mainconj}

However, by looking at the functional equation satisfied by the shifted $q$-metallic numbers $\Phi_n^{(\ell)}$ and applying a result of \cite{Han16}, we show that weak versions of \cref{MainHF} and \cref{MainValPer} hold, even in case $\ell≥n+2$:

\begin{mainthm}[{Periodicity modulo $p$}]\label{MainModulop}
Let $p$ be a prime number. For any $\ell≥0$, the Hankel continued fraction of the shifted $q$-metallic number $\Phi_n^{(\ell)}$ and the corresponding Hankel determinants $\De_j^{(\ell)}$ ($j≥0$) are ultimately periodic modulo $p$.
\end{mainthm}

\medbreak \noindent \textbf{Organization of the paper}. 
\cref{sec-bg} contains a brief presentation of  S.~Morier-Genoud and V.~Ovsienko's theory of $q$-deformation of real numbers, applied in particular to metallic ratios. We recall also from \cite{Han16} some fundamental results on Hankel continued fractions (alias $H$-fractions), since it is the main tool we will use to compute Hankel determinants. We also explain how close Hankel continued fractions are to Artin's regular continued fractions in the context of Laurent series.

In \cref{sec-HF} we prove \cref{MainHF} i.e. the $H$-fraction expansion of $\Phi_n$. It will be obtained by using an algorithm well suited to power series that are solution of a quadratic equation with polynomial coefficients. This algorithm appeared first in the paper \cite{Han16} and its application to our $q$-metallic numbers $\Phi_n$ represents a quite technical work.

\cref{MainValPer,MainGR} will be obtained with the following strategy: instead of proving them directly for all values of the parameter $\ell$, we note that cases $1≤\ell≤n+1$ follow at once from the case $\ell=0$ and \cref{MainContiguity}. This is why \cref{section-period} is devoted to the proof of \cref{MainValPer}  in the case $\ell=0$ only. The main tool here is the Hankel determinant formula for $H$-fractions recalled in \cref{sec-bg}.
It is worth noting that \cref{MainValPer} will be established without calculating explicitly the Hankel determinants, since for showing periodicity and specifying the image set we need to understand more the pattern they form than their exact values. 

Yet, computing the Hankel determinants $\De_j=\De_j^{(0)}(\Phi_n)$ is inevitable for our other results and it is done in \cref{section-HD0}. We exploit the (anti-)periodicity shown by \cref{MainValPer} together with some extra symmetry property to restrict ourselves to calculate  the  values of $\De_{j}$ only for a half-period (more or less), and we obtain \emph{explicit formulas} similar to the ones in \eqref{HankelGold2}. They are presented in \cref{thm-Hankel0}.

From these exact values of the Hankel determinants we then derive their Gale-Robinson dynamics and prove  \cref{MainGR} in \cref{sec-GR}, still in the case $\ell=0$. This is an elementary but quite tedious task.

Finally, \cref{section-shifts} is devoted to shifted Hankel determinants $\De_j^{(\ell)}$ with $\ell≥1$. We start by proving \cref{HFPhishifts} which expresses the $H$-fraction expansion  of the shifted functions $\Phi_n^{(\ell)}$  ($1≤\ell≤n+1$) in terms of the one given in \eqref{Hfrac-Phin} for the initial function $\Phi_n=\Phi_n^{(0)}$. This result interconnects the $H$-fractions expressions of  all functions $\Phi_n^{(\ell)}$, $0≤\ell≤n+1$ (see \cref{contiguityPhin}) and leads to the proof of \cref{MainContiguity} which, as noticed before, implies eventually the statements of \cref{MainValPer,MainGR} in the remaining cases $1≤\ell≤n+1$. In this section we also prove \cref{MainModulop}.

Let us add a comment to our presentation. Although the cases of $\Phi_1$ and $\Phi_2$ were treated in \cite{OP25} with much detail and very explicit formulas, they do not really represent the general case $\Phi_n$ which is more complex in nature. Moreover, our proofs are quite technical and lengthy. We thus think it is worthwhile to  illustrate our present  results on another concrete case, more representative of the general case but not too complicated either. We thus have chosen to focus on the case $n=5$ throughout the paper to examplify the main results and the more technical proofs.

\medbreak \noindent
\textbf{Note.} Since all results stated in this article were  proved in the cases $n=1$ and $n=2$ in \cite{OP25}, from now on it will be convenient to always assume that $n≥3$,  in order to avoid special discussions for the limit case $n=2$ in some results.

\medbreak \noindent
\textbf{Acknowledgements.} 
We are grateful to Valentin Ovsienko for his stimulating interest and his useful comments on this work.


\section{Background material}\label{sec-bg}

\subsection{S.~Morier-Genoud and V.~Ovsienko's $q$-numbers} \label{sec-qreals}

Our article is devoted to the continuous fraction expansion and  Hankel determinant properties of the \emph{$q$-deformation} of particular real numbers. Let us define precisely what kind of quantization is meant here and give a brief introduction to the subject.

Quantization of integers is a very classical topic which goes back to Euler and Gauss:  any integer $n≥0$ can be deformed as a polynomial
\begin{equation}\label{EGEq}
[n]_q:=\frac{1-q^n}{1-q}=1+q+q^2+\cdots+q^{n-1}.
\end{equation}
Equivalently, one can define $[n]_q$ as the unique solution of the recurrence 
\begin{equation}\label{rec}
[n+1]_q=q[n]_q+1,\qquad [0]_q=0.
\end{equation}
First form of definition \eqref{EGEq} can also be used to define $q$-deformation of negative integers $n<0$:
$$[n]_q:= \frac{1-q^n}{1-q}=-q^{-1}-q^{-2}-\cdots -q^{n}$$
and this gives a polynomial in $q^{-1}$.
A considerable amount of related objects have been based on definition \eqref{EGEq}, e.g. $q$-factorials, $q$-binomials, $q$-hypergeometric functions, $q$-calculus, etc. used in combinatorics, number theory, fractals, operator theory, mathematical physics… But until recently we missed a really satisfactory extension to more general numbers (reals, or even just rationals). For instance, for a rational $m/n$, the naïve ideas
$$\left[\fr{m}{n}\right]_q:=\fr{[m]_q}{[n]_q}
\qquad\text{or}\qquad
\left[\fr{m}{n}\right]_q:=\frac{1-q^{m/n}}{1-q}$$ 
both lead to the same notion (up to a rescaling) which lacks crucial properties, such as verifying \eqref{rec} or being a rational function of $q$; it is then difficult to have nice enumerative or geometric interpretations.

In 2018, S.~Morier-Genoud and V.~Ovsienko began a beautiful construction of  
 $q$-analogues for rational, real and even complex numbers (see the seminal papers \cite{MGO20,MGO22,Ovsienko21}) which in a very few years has shown to be  successful in the mathematical community, thanks to connections with many  topics: enumerative combinatorics, cluster algebras, Markov-Hurwitz approximation theory, braid groups, combinatorics of posets, Calabi-Yau triangulated categories, Lie algebras of differential operators, supersymmetry and supergeometry… see e.g. \cite{MSS21,BBL23,CO23,KOM23,MO24,MGOV24,Thomas24,KMRWY25,Jouteur25}.

Several equivalent models are available for Morier-Genoud \& Ovsienko's $q$-real numbers. One of them involves continued fractions and is the best suited for our purposes.

Let $x\in\R$ and consider its regular continued fraction expansion (finite if and only if $x\in\Q$)
$$
x\;=\;
a_0 + \frac{1}{a_1}\cfp\frac{1}{a_2}\cfp\frac {1}{a_3}\cfp\cfd
$$
with $a_i$ integers, positive for $i≥1$. The \emph{$q$-deformation} or \emph{$q$-analogue} of~$x$ is  the following algebraic continued fraction:
\begin{equation}\label{defq}
[x]_q:=[a_0]_{q} + 
	\frac{q^{a_{0}}}{[a_1]_{q^{-1}}}\cfp \frac{q^{-a_{1}}}{[a_{2}]_{q}}\cfp
	\frac{q^{a_{2}}}{[a_3]_{q^{-1}}}\cfp \frac{q^{-a_{3}}}{[a_{4}]_{q}}
\cfp\cfd
\end{equation}
where $[n]_q$ stands for the $q$-integer as in~\eqref{EGEq}, and $[n]_{q^{-1}}=q^{1-n}[n]_q$ is the same expression with reciprocal parameter.
Note that, when infinite, the right-hand side in \eqref{defq} always converges in the sense of formal Laurent series.

Let us indicate some important characteristics of these $q$-numbers; we cite the original references but most of the properties are explained also in  the short survey \cite{Ovsienko23}. 
\begin{enumerate}
\item {\cite{MGO20}} For $x\in\Z$, $[x]_q$ coincides with the classical definition \eqref{EGEq} of Euler and Gauss.
\item {\cite{MGO22}} For any $x\in\R$, the Laurent series $[x]_q$ has integer coefficients. More precisely the quantification map $[\,\cdot\,]_q$ has the following images
\begin{equation*}
[\,\cdot\,]_q:\quad \Q_{≥0}\to\Z_{≥0}(q),\quad \Q\to\Z(q),\quad \R_{≥0}\to\Z[[q]],\quad \R\to\Z(\!(q)\!).
\end{equation*} 
\item {\cite{MGO22,LMG21}} We have, for any $x\in\R$,
\begin{equation*}
[x+1]_q=q[x]_q+1,\qquad \left[-\fr 1{x}\right]_q=-\fr 1{q[x]_q}.
\end{equation*}
These two crucial properties  can be interpreted as commutation of the quantification map $[\,\cdot\,]_q$  with the action on $\Z(\!(q)\!)\cup\{\infty\}$ of a group of linear-fractional transformations which is isomorphic to the modular group $\PSL(2,\Z)$. Moreover, $[\,\cdot\,]_q$ is the unique $\PSL(2,\Z)$-invariant quantification which fixes $0$; see \cite{OP25}, §2.2. Such an invariance is at the heart of the construction and plays a crucial role in the proof of numerous properties of these $q$-numbers. Also, it can  be understood as invariance with respect to the Burau representation of
the braid group $B_3$ \cite{BBL23,MGOV24};
\item The radius of convergence of the Taylor series at $0$  of  the $q$-golden number $\left[\fr{1+\sqrt{5}}{2}\right]_q=\Phi_1(q)$ equals $\fr{3-\sqrt{5}}{2}$, and it is conjectured that the radius of convergence of the Taylor series at $0$  of the $q$-deformation of any positive real number is always $≥\fr{3-\sqrt{5}}{2}$  \cite{LMGOV24}. This statement has already been verified for all $q$-metallic numbers $\Phi_n(q)$ (see \cite{Ren22,Ren23}) and for a large class of $q$-deformed positive quadratic irrationals (see \cite{Leclere24}, §4.4). 
For general rational or real numbers, progress has been made recently, see \cite{EGMS,Etingof25}.
\item {\cite{LMG21}} When $x=\fr{r\pm\sqrt{p}}{s}$ is a quadratic irrational  number ($p,r,s$ integers, $p>0$), $[x]_q$ is itself of the form 
\begin{equation}\label{q-quad}
[x]_q=\fr{R(q)\pm \sqrt{P(q)}}{S(q)}
\end{equation}
with $P,R,Q\in\Z[q]$ and $P$ a palindrome.
\end{enumerate}

Examples illustrating formula \eqref{q-quad} are yielded by the  $q$-deformation  $\Phi_n(q)$ of the metallic numbers $\phi_n=\fr{n+\sqrt{n^2+4}}{2}$, $n≥1$. Indeed,  using definition \eqref{defPhin} one can see that $\Phi_n$ is characterized by the functional equation
\begin{equation}\label{FE-Phin}
q\, \Phi_n(q)^2+\left((1+q^n)(1-q) - q\left[n\right]_q\right)\Phi_n(q)=1,
\end{equation}
from which follows at once a formula of  type \eqref{q-quad}:
\begin{equation*}
	\Phi_n(q)=\frac{1}{2q}\left(q[n]_q+(q^n+1)(q-1)+\sqrt{\Bigl(q[n]_q+(q^n+1)(q-1)\Bigr)^2+4q}\right).
\end{equation*}
See \cite{LMG21} or Proposition~2.5 in \cite{OP25} for details. Another important property we will use is that $\Phi_n$ admits the following Taylor series expansion about zero:
\begin{equation}\label{SE-Phin}
\Phi_n(q)=1+q+\cdots+q^{n-1}+q^{2n}+\sum_{i=2n+1}^{+\infty}\kappa_i q^i,
\end{equation}
where coefficients $\kappa_i$ are integers; see Corollary~2.6 in \cite{OP25}. 

\begin{ex} The power series expansion of $\Phi_1(q)$ was given in \eqref{SEPhi1}. Let us write also the expansions of $\Phi_2(q)$, $\Phi_5(q)$ and $\Phi_{10}(q)$, respectively:
\begin{align*}
\Phi_2(q)&=1+q+q^4-2q^6+q^7+4q^8-5q^9-7q^{10}+ 18q^{11}+ 7q^{12}-55q^{13}+ 18q^{14}\\
&\quad+ 146q^{15}- 155q^{16} - 322q^{17}+692q^{18}+ 476q^{19}- 2446q^{20}+ 307 q^{21} + 7322 q^{22} + \cdots\\
\Phi_5(q)&=1 +  q + q^{2} + q^{3} + q^{4} + q^{10} - q^{12} - q^{13} + 3 q^{16} + 3 q^{17}-2 q^{18} -7 q^{19}\\
&\quad -4q^{20} - q^{21} + 10 q^{22} + 21 q^{23} + 9 q^{24} -30 q^{25} -44 q^{26} -28q^{27} + 27 q^{28} + 115 q^{29}+ \cdots\\
\Phi_{10}(q)&=1 + q + q^{2} + q^{3} + q^{4} + q^{5} + q^{6} + q^{7} + q^{8} + q^{9} + q^{20} - q^{22} - q^{23} +  q^{25}
\\
&\quad +  q^{26} - q^{28} - q^{29} - q^{30} + 3 q^{31} + 4 q^{32} - q^{33} -7 q^{34}
-6 q^{35} + 3 q^{36} + 11 q^{37} + 8 q^{38}+\cdots
\end{align*}
\end{ex}

\subsection{$H$-fractions and their Hankel determinants} 
\label{secH-frac}
In this subsection we recall some material from \cite{Han16}. Fix a field $\F$, an indeterminate $q$, and let $\de$ be a positive integer. A \emph{super $\de$-fraction} is a continued fraction of the form
\begin{equation}
\label{superfrac}
F(q)=\fr{v_0\,q^{k_0}}{1+q u_1(q)}\cfp\fr{-v_1\,q^{k_0+k_1+\de}}{1+qu_2(q)}\cfp
\fr{-v_2\,q^{k_1+k_2+\de}}{1+q u_3(q)}\cfp
\fr{-v_3\,q^{k_2+k_3+\de}}{1+q u_4(q)}\cfp\cfd
\end{equation}
where $v_i\in \F^*$, $k_i\in\Z_{≥0}$, and $u_i\in\F[q]$ are polynomials such that $\deg(u_i)≤k_{i-1}+\de-2$.

One of the main interests of this notion is that, not only \eqref{superfrac} gives always a power series, but also, conversely, for a given positive integer $\de$, any power series admits one and only one super $\de$-fraction expansion (finite if and only if $F$ is rational). We will need to recall the proof of this result at the beginning of \cref{section-algo}.

\begin{ex}\label{ex-3f-Phin} According to Theorem~3.2 in \cite{OP25}, for any $n≥1$ the $q$-deformed metallic number $\Phi_n$ admits the following $1$-periodic expression as a super $3$-fraction:
\begin{equation*}
\Phi_n(q)
= [n]_q+\fr{q^{2n}}{\cron}\cfp\left(\fr{q^{2n+1}}{\cron}\cfp\right)^*
\end{equation*}
where $\cron$ was defined in \eqref{defnq}. To be accurate, it is the shifted function 
$$\Phi_n^{(n+1)}(q)=\frac{\Phi_n(q)-[n]_q}{q^{n+1}}$$
(see definition \eqref{defFl}) which  matches exactly definition \eqref{superfrac} of a super $3$-fraction, with $k_j\equiv n-1$ for $j≥0$, $v_0=1$ and $v_j\equiv -1$ for $j≥1$.
\end{ex}

Super $\de$-fractions with $\de=2$ are of special importance, they are called \emph{Hankel continued fractions}, or more simply \emph{$H$-fractions}. Firstly, they generalize the classical notion of regular $C$-fraction and regular $J$-fraction (when all $k_i$ are zero and $\de=1$, $\de=2$, respectively) and offer the possibility to compute the Hankel determinants of $F$ with the following striking formula (see Theorem 2.1 of~\cite{Han16}): if we define
\begin{equation}
\label{s-def}
s_p:=p+\sum_{i=0}^{p-1}k_i
\quad\text{if }p≥1,
\quad s_0:=0,
\quad\cS:=\{s_p,\ p≥0\}
\end{equation}
and
\begin{equation}
\label{e-def}
\eps_p:=\sum_{i=0}^{p-1}\frac{k_i\left(k_i+1\right)}{2}
\quad\text{if }p≥1
\quad\text{and}\quad \eps_0:=0,
\end{equation}
then 
\begin{equation}
\label{HankelHF}
\begin{dcases*}
\Delta_{s_p}(F)
=(-1)^{\eps_p}v_0^{s_p}v_1^{s_p-s_1}v_2^{s_p-s_2}\cdots v_{p-1}^{s_p-s_{p-1}}&if $p≥1$,\\\Delta_j\left(F\right)=0 &if $j\notin\cS$.
\end{dcases*}
\end{equation}
(Recall that, by convention above, $\De_{s_0}(F)=\De_0(F)=1$.) In this way, classical formulas known for $C$-fractions and $J$-fractions (see e.g. Theorems~29 and~30 in \cite{Krattenthaler05}) are generalized. Let us mention that similar formulas to \eqref{HankelHF} had been studied earlier in several references (mostly independently); see \cite{Han20}, the historical remark after Theorem~2.2.

Besides formula  \eqref{HankelHF}, a second reason for which we think that Hankel continued fractions have a particular interest is their strong connection with a classical continued fraction expansion in the field $\F(\!(q)\!)$ of Laurent series. Let us elaborate, since this relation was not yet discovered in the  paper \cite{Han16}. 

E.~Artin  has shown \cite{Artin24} that any Laurent series in $q$
$$
f(q)= f_{-n} q^{-n}+ \cdots+ f_{-1} q^{-1}+  f_0+ f_{1} q+ f_{2} q^{2}+ \cdots 
$$
can be expanded, in a unique manner, as a continued fraction of the form
\begin{equation}\label{Artin}
f(q)= a_0(1/q)+\fr 1{a_1(1/q)}\cfp\fr 1{a_2(1/q)}\cfp\fr 1{a_3(1/q)}\cfp\cfd
\end{equation}
where $a_0$ is a polynomial and, for any $j≥1$, $a_j$ is a non constant polynomial. Artin's expansion is traditionnaly called the \emph{regular continued fraction expansion} of $f$, because of the familiarity with the classical notion of regular continued fraction expansion for real numbers\footnote{Usually, Artin's theorem is stated in the field $\F(\!(1/q)\!)$ of Laurent series in the variable $1/q$ but we adapt it to our context.}. 

It turns out that  regular continued fractions and Hankel continued fractions are closely and simply related when we restrict to power series. Take such a function
$$f(q)= f_{1} q+f_{2} q^{2}+f_3 q^3+ \cdots 
$$
\emph{with no constant term}
and write its regular expansion \eqref{Artin}:
\begin{equation}\label{ArtinPS}
f(q)=\fr 1{a_1(1/q)}\cfp\fr 1{a_2(1/q)}\cfp\fr 1{a_3(1/q)}\cfp\cfd
\end{equation}
Fix $j≥1$. Since $a_j$ is a non constant polynomial, there exist two natural numbers $m_j$ and $n_j$, with $1≤n_j≤m_j$, and a family of scalars $(c_{j,r})_{r=n_j}^{m_j}$, with $c_{j, n_j}\not=0$ and $c_{j, m_j}\not=0$, such that
\begin{equation*}
a_j(1/q) 
=\frac {c_{j,n_j} }{q^{n_j}} +  \frac {c_{j,n_j+1}}{q^{n_j+1}} + \cdots +  \frac {c_{j, m_j}}{q^{m_j}}.
\end{equation*}
We then write
\begin{equation*}
a_j(1/q) 
=\bigl(1 + u_j(q)q\bigr)  \frac {c_{j, m_j}}{q^{m_j}}
\end{equation*}
where
\begin{equation*}
u_j(q) 
=   \frac{c_{j,n_j}}{c_{j,m_j}} q^{m_j-n_j-1} + \frac{c_{j,n_j+1}}{c_{j,m_j}} {q^{m_j-n_j-2}} + \cdots+ \frac{c_{j,m_j-1}}{c_{j,m_j}}
\end{equation*}
is a polynomial in $q$ satisfying  
\begin{equation}\label{deg-uj}
\deg(u_j)= m_j - n_j-1≤m_j-1.
\end{equation}
But then \eqref{ArtinPS} reads
\begin{align*}
f(q)
&=\frac{1}{ (1 + u_1(q)q)  \frac {c_{1, m_1}}{q^{m_1}}}
\cfp \frac{1}{ (1 + u_2(q)q)  \frac {c_{2, m_2}}{q^{m_2}}}
\cfp  \frac{1}{ (1 + u_3(q)q)  \frac {c_{3, m_3}}{q^{m_3}}}
\cfp\cfd\\
&=\frac{c_{1, m_1}^{-1}{q^{m_1}}}{ 1 + u_1(q)q}
\cfp \frac{c_{1, m_1}^{-1} c_{2, m_2}^{-1}{q^{m_1+m_2}}}{ 1 + u_2(q)q}
\cfp \frac{c_{2, m_2}^{-1} c_{3, m_3}^{-1}{q^{m_2+m_3}}}{ 1 + u_3(q)q}
\cfp\cfd
\end{align*}
This expression is exactly of the form
\begin{equation*}
f(q)=q\,F(q)
\end{equation*}
where 
\begin{equation*}
F(q)=\fr{v_0\,q^{k_0}}{1+q u_1(q)}\cfp\fr{-v_1\,q^{k_0+k_1+2}}{1+qu_2(q)}\cfp
\fr{-v_2\,q^{k_1+k_2+2}}{1+q u_3(q)}\cfp\cfd
\end{equation*}
and
\begin{equation}\label{dictionary}
\begin{dcases*}
v_0=c_{1, m_1}^{-1}&\\
v_j=-\left(c_{j, m_j} c_{j+1, m_{j+1}}\right)^{-1}&for $j≥1$\\
k_j=m_{j+1}-1&for $j≥0$.
\end{dcases*}
\end{equation}
Since $\deg(u_j)≤k_{j-1}$ by \eqref{deg-uj}, the expression of $F(q)$ above is a $H$-fraction expansion (take \eqref{superfrac} with $\de=2$) for the parameters \eqref{dictionary}.

In conclusion, Artin's regular continued fraction expansion for the power series without constant term $f(q)$ yields a Hankel continued fraction expansion for the power series $F(q)=f(q)/q$. And conversely, the Hankel continued fraction expansion for any power series $g(q)$ yields the regular continued fraction expansion for $(g(q)-g(0))/q$. We thus have a correspondance between the two kinds of expansion and it is easy to switch from one model to the other, thanks to the conversion formulas \eqref{dictionary}.

As an example, consider a regular $J$-fraction expansion
\begin{equation*}
F(q)=\fr{v_0}{1+ u_1 q}\cfp\fr{-v_1\,q^{2}}{1+u_2 q}\cfp
\fr{-v_2\,q^{2}}{1+ u_3q}\cfp\cfd
\end{equation*}
where $u_i,v_i$ are scalars and $v_i$ are nonzero. It is nothing but a $H$-fraction with all $k_j=0$, and if $$f(q)=f_0+qF(q),$$ then $f(q)$ admits a regular continued fraction \eqref{ArtinPS} with all $a_j$ polynomials of degree $1$ and determined by the parameters of the $J$-fraction as follows: 
\begin{equation*}
f(q)=f_0+\frac{1}{c_1+\fr{d_1}{q}}\cfp\frac{1}{c_2+\fr{d_2}{q}}\cfp\frac{1}{c_3+\fr{d_3}{q}}\cfp\cfd
\end{equation*}
with
\begin{equation*}
\begin{dcases*}
d_j≠0&for $j≥0$\\
v_0=d_1^{-1}&\\
v_j=-(d_j d_{j+1})^{-1}&for $j≥1$\\
u_j=c_j / d_j&for $j≥0$.
\end{dcases*}
\end{equation*}
Such a transformation for regular $J$-fractions may be considered as folklore and not necessarily connected to Artin's expansion; it appears e.g. in \cite{Wall}, §42.

\section{\texorpdfstring{The Hankel continued fraction expansion for $\Phi_n$}{The Hankel continued fraction expansion}}
\label{sec-HF}

We already recalled the examples of $H$-fractions \eqref{HFGold} and \eqref{HFSilver} that were obtained for $\Phi_1(q)$ and $\Phi_2(q)$ in \cite{OP25}. In this section we prove that the $H$-fraction presentation for $\Phi_n$ stated in \cref{MainHF} is actually valid for all $n≥3$. It is a quite difficult problem, and we will solve it by applying a variant of an algorithm which was  introduced in \cite{Han16}. It was originally designed to treat quadratic power series in coefficients in a finite field, but actually works perfectly for any field $\F$ and in particular in our situation $\F=\Q$.

\subsection{An algorithm for $H$-fractions}
\label{section-algo}

Let $\F$ be a field, $q$ an indeterminate. Here we present an algorithm which is tailored to produce explicit $H$-fractions for a certain class of power series $F\in\F[[q]]$, including our particular function $\Phi_n\in\Q[[q]]$. 
 
To begin with, let us explain how any power series $F\in\F[[q]]$ can we written, in a unique way, as a $H$-fraction, i.e. a super $\de$-fraction \eqref{superfrac} with $\de=2$ (actually the same argument works for any $\de$, see the proof of Theorem~2.1 in \cite{Han16}).
 
First, one roughly expands the power series $F(q)$ as
\begin{equation*}
F(q)=v_0 q^{k_0}+O(q^{k_0+1})
\end{equation*} 
where $v_0≠0$ and $k_0\in\Z_{≥0}$, so that $F(q)/(v_0 q^{k_0})=1+O(q)$. Hence one can write
\begin{equation}\label{defD}
\fr{v_0 q^{k_0}}{F(q)}=D_0(q)-q^{k_0+2}F_{1}(q)
\end{equation}
with  $D_0$ a polynomial such that $D_0(0)=1$ and $\deg(D_0)≤k_0+1$, and $F_{1}(q)$ a power series. Notice that the condition on the degree makes $D_0$ and $F_1$ uniquely defined. Equivalently we have
\begin{equation}\label{defD2}
F(q)=\fr{v_0 q^{k_0}}{D_0(q)-q^{k_0+2}F_{1}(q)}
\end{equation}
and repeating the process with $F_1$ will give similar objects $v_1,k_1,D_1,F_2$ such that 
\begin{equation*}
F_1(q)=\fr{v_1 q^{k_1}}{D_1(q)-q^{k_1+2}F_{2}(q)}
\end{equation*}
hence
\begin{equation*}
F(q)=\fr{v_0 q^{k_0}}{D_0(q)}\cfp\fr{-v_1 q^{k_0+k_1+2}}{D_1(q)-q^{k_1+2}F_{2}(q)}.
\end{equation*}
Going on with this process inductively will lead to the Hankel continued fraction \begin{equation}\label{Hfexistence}
F(q)=\fr{v_0 q^{k_0}}{D_0(q)}\cfp\fr{-v_1 q^{k_0+k_1+2}}{D_1(q)}\cfp
\fr{-v_2 q^{k_1+k_2+2}}{D_2(q)}\cfp\cfd
\end{equation} 

Thus we have proved that any power series $F(q)$ can be developed (uniquely) as a $H$-fraction, and next problem is to compute explicitly  data sequences $(v_i)$, $(k_i)$ and $(D_i)$ from $F$. We do not know the answer in general, but when the power series is \emph{quadratic} in the sense that there exists three polynomials $A,B,C\in\F[q]$ such that
\begin{equation}\label{Fquad}
A+BF+CF^2=0
\end{equation}
(together with some conditions on $A,B,C$ that will be described soon) an algorithm will help us to do that. 
Of course, it will only be usable in  case of a finite number of steps. This situation can happen when sequences of data are finite by nature (e.g. when the field is finite; see the original applications in \cite{Han16}) or  when the super fraction is going to be \emph{periodic} (or \emph{ultimately periodic}), which will be our case when considering $F=\Phi_n$, although the field $\F=\Q$ is infinite.

Let us elaborate. Suppose from now on that $F$ satisfies the quadratic equation \eqref{Fquad} together with the following conditions
\begin{equation}
A≠0,\quad B(0)=1,\quad C≠0,\quad C(0)=0.\label{condABC}
\end{equation}
(Recall from \eqref{FE-Phin} that this is the case for $F=\Phi_n$.) Then the quadratic equation \eqref{Fquad} with conditions \eqref{condABC} has a unique solution
\begin{equation*}
F=\fr{-B+\sqrt{B^2-4AC}}{2C}.
\end{equation*}
According to \eqref{defD} we need to know the first term of the expansion of $F(q)$ about $q=0$. Because
\begin{equation*}
F(q)=\fr{-A(q)}{B(q)+C(q)F(q)}
\end{equation*}
with conditions \eqref{condABC}, we see that this amounts to know the lower degree term of $A(q)$. Writing the nonzero polynomial $A$  as
\begin{equation}\label{defak}
A(q)=a_k q^k + O(q^{k+1})
\end{equation}
with $k\in\Z_{≥0}$ and $a_k\not=0$, the  left-hand side of \eqref{defD} reads
\begin{align}
\frac{- a_k q^k} {F(q)} 
&=\frac{-2a_k q^k C}{-B+\sqrt{B^2-4AC}}\notag\\
&=\frac{-2a_k q^k C(\sqrt{B^2-4AC}+B)}{-4AC}\notag\\
&=\frac{-a_kq^kC}{A} \times \frac{-B-\sqrt{B^2-4AC}}{2C}.\label{eicu}
\end{align}
(From now on, for polynomials we will often write $P$ instead of $P(q)$, for easier reading.) Let
\begin{equation*}
\Ga(q)=\sum_{j=0}^{+\infty}\ga_j q^j=\fr{1-\sqrt{1-4q}}{2q}
\end{equation*}
be the generating function of the Catalan numbers \eqref{Catalan}. We calculate
\begin{equation*}
\Ga\left(\fr{AC}{B^2}\right)=\fr{B}{A}\times\fr{B-\sqrt{B^2-4AC}}{2C}
\end{equation*}
which implies
\begin{align*}
\frac{-B-\sqrt{B^2-4AC}}{2C}
&=-\fr{B}{C}-\frac{-B+\sqrt{B^2-4AC}}{2C}\\
&=-\fr{B}{C}+\fr{A}{B}\:\Ga\left(\fr{AC}{B^2}\right)\\
&=-\fr{B}{C}+\fr{A}{B}\sum_{j=0}^{+\infty} \gamma_j\left(\fr{AC}{B^2}\right)^j.
\end{align*}
Inserting this equality in \eqref{eicu} we deduce that
\begin{align*}
\frac{- a_k q^k} {F(q)} 
&=\frac{-a_kq^kC}{A}  
	\biggl(-\fr{B}{C} + \fr{A}{B}\sum_{j=0}^{+\infty} \gamma_j \left(\fr{AC}{B^2}\right)^j\biggr)\\
&=\frac{-a_kq^kB}{A}  
	\biggl(-1 + \sum_{j=0}^{+\infty} \gamma_j \left(\fr{AC}{B^2}\right)^{j+1}\biggr).
\end{align*}
Now, as explained in \eqref{defD} there exists  a unique polynomial $D$ with $D(0)=1$ and $\deg(D)≤k+1$, and  a unique power series $R_1$ such that 
\begin{equation*}
\fr{-a_k q^k}{F(q)}=D-q^{k+2}R_1(q)
\end{equation*}
hence
\begin{equation}\label{defD3}
\frac{-a_k q^k B}{A} \biggl(-1 + \sum_{j=0}^{+\infty} \gamma_j \left(\fr{AC}{B^2}\right)^{j+1}\biggr)
=D(q)  -  q^{k+2 }R_1(q).
\end{equation}
Recall that we want to compute explicitly the polynomial $D$ since it will produce the denominator in the first step \eqref{defD2} of the $H$-fraction expansion of $F$. To do so, remind from \eqref{condABC} that we have $B=1+O(q)$ and $C=O(q)$. Together with \eqref{defak} this implies
\begin{equation*}
\fr{AC}{B^2}=\fr{[a_k q^k+O(q^{k+1})]\cdot O(q)}{[1+O(q)]^2}=a_kO(q^{k+1})+O(q^{k+2})
\end{equation*}
so that $\left(\fr{AC}{B^2}\right)^{j+1}=O(q^{k+2})$ for all $j≥1$. In other words
the series in the left-hand side of \eqref{defD3} can contribute to $D$ only with its first term, and there exists a power series $R_2$ such that
\begin{equation*}
\frac{-a_kq^kB}{A} \left(-1 + \frac{AC}{B^2}\right)=D - q^{k+2 }R_2(q).
\end{equation*}
Writing $C(q)=c_1q+c_2 q^2+\cdots+c_d q^d$ and remembering $B=1+O(q)$, we conclude that $D$ is uniquely determined by the following conditions:
\begin{equation}\label{defD4}
\deg(D)≤k+1\quad\text{and}\quad a_k q^k\fr{B}{A}-a_k c_1 q^{k+1}=D- q^{k+2 }R_3(q)
\end{equation}
for some power series $R_3$. Property that $D(0)=1$ will be automatic, by \eqref{defD}.

All of this leads us to formulate the following definition, already introduced in \cite{Han16}, but refined here in the special case $\de=2$, i.e. when super fractions are $H$-fractions.

\begin{algo}\label{alg}{} [Algorithm \alg for $H$-fractions] 
\begin{description}
\item [Prototype] $(A^*, B^*, C^*, k, a_k, D)=\alg(A,B,C)$.
\item [Input] three polynomials $A, B, C\in \F[q]$  satisfying \eqref{condABC}.
\item [Output] a positive integer $k$,\\
a nonzero scalar $a_k\in\F$,\\
 a polynomial $D\in\F[q]$ of degree $≤k+1$ such that $D(0)=1$,\\
three polynomials $A^*, B^*, C^*\in\F[q]$ satisfying  \eqref{condABC}.  
\item [Step 1] Define $k, a_k$ by \eqref{defak}.
\item [Step 2] Define $D$ by \eqref{defD4}.
\item [Step 3] Set
\begin{align*}
{A^*}&:=\fr {1}{q^{2k+2}}\left(\fr{-D^2A}{a_k}+BD q^k-Ca_kq^{2k}\right)\\
	{B^*}&:=\fr{2AD}{a_kq^{k}} - B\\
	{C^*} &:=-\fr{Aq^{2}}{a_k}.
\end{align*}
\end{description}
\end{algo} 

In this algorithm, all objects are well (and uniquely) defined by the previous discussion.
It remains only to justify that $A^*,B^*,C^*$ are indeed polynomials and that they also  satisfy \eqref{condABC}; but this comes from Lemma~3.2 in \cite{Han16}. 
\medbreak

Now, let us check that our algorithm \alg works as expected:
\begin{itemize}
\item Start with a quadratic power series $F$ as in \eqref{Fquad}, with the associated triple of polynomials $(A,B,C)$ satisfying \eqref{condABC}. 
\item Set $(A_1,B_1,C_1,k_0,a_0,D_0):=\alg(A,B,C)$. Then we have
\begin{equation*}
F(q)=\fr{-a_0 q^{k_0}}{D_0(q)-q^{k_0+2}F_{1}(q)}
\end{equation*}
as in \eqref{defD2}. The main point now is that the power series $F_1$ is itself quadratic, more precisely it satisfies the equation
\begin{equation*}
A_1+B_1F_1+C_1F_1^2=0.
\end{equation*}
This fact is also a statement taken from Lemma~3.2 in \cite{Han16}. 
\item Thus the algorithm can be applied a second time and we set $(A_2,B_2,C_2,k_1,a_1,D_1):=\alg(A_1,B_1,C_1)$. As in the beginning of this subsection we have
\begin{equation*}
F_1(q)=\fr{-a_1 q^{k_1}}{D_1(q)-q^{k_1+2}F_{2}(q)}
\qquad\text{hence}\qquad
F(q)=\fr{-a_0 q^{k_0}}{D_0(q)}\cfp\fr{a_1 q^{k_0+k_1+2}}{D_1(q)-q^{k_1+2}F_{2}(q)}
\end{equation*}
with $F_2$ still quadratic: $A_2+B_2F_2+C_2F_2^2=0$. Going on through this  process will define $(k_j,a_j,D_j)$ for any $j≥0$ and will lead to the aimed $H$-fraction \eqref{Hfexistence}.
\item Suppose that there exists integers $0<p<p'$ such that $(A_{p'},B_{p'},C_{p'})=(A_p,B_p,C_p)$. Then the algorithm will repeat the same sequences indefinitely and the $H$-fraction is ultimately periodic:
\begin{equation}\label{Hfracalgo}
F(q)=\fr{-a_0  q^{k_0}}{D_0(q)}\cfp\CF_{j=1}^{p-1}\fr{a_j q^{k_{j-1}+k_j+2}}{D_j(q)}\cfp
\left(\CF_{j=p}^{p'-1}\fr{a_j q^{k_{j-1}+k_j+2}}{D_j(q)}\right)^*.
\end{equation} 
This is precisely what is going to happen when $F=\Phi_n$.
\end{itemize}

Of course, in theory it is not necessary to define $D$ by \eqref{defD4} rather than by \eqref{defD} to get the $H$-fraction; the original algorithm in \cite{Han16}, constructed to produce super $\de$-fractions for any $\de$, does not do that. But in our case $\de=2$, it will be much more easier to find the polynomials $D_j$ by looking  at simple expressions involving only polynomials $A_{j},B_{j},C_{j}$. 

\begin{rem} Actually, \cref{alg} can be simplified. Indeed, because of definition of $C^*$ in Step~3 we see that \emph{we will always have $C=O(q^2)$, except maybe at first implementation of the algorithm}. In other words $c_1=0$ and definition \eqref{defD4} of the polynomial $D$ in Step~2 of \cref{alg}  becomes simpler.
\\
Actually, we will encounter further simplifications when applying \cref{alg} to $F=\Phi_n$: the coefficients $a_k$ will always be $\pm 1$, as we shall see soon.
\end{rem}

\subsection{Proof of \texorpdfstring{\cref{MainHF}}{Theorem~\ref{MainHF}}}
\label{section-pfD}

Fix an integer $n≥3$. In this subsection we will apply \cref{alg} to $F=\Phi_n$ and prove \cref{MainHF}. As explained after the statement of \cref{alg}, our task is  to exhibit a  family of sextuplets associated with $F=\Phi_n$:
\begin{equation}\label{family}
\mathcal{Q}_j:=(A_j,B_j,C_j,k_j,a_j,D_j),\quad j≥0,
\end{equation} 
by applying repeatedly the algorithm, until periodicity is observed and stops the process. Then the $H$-fraction will have the form \eqref{Hfracalgo}.

In fact, it turns out that we will need to define two different families of objects to achieve our goal. We start with the first one.

\begin{defi}\label{defi-A_1}  
Let $\bigl(\mathcal{Q}_1^{(m)}\bigr)_{m≥1}$ denote the sequence of sextuplets
\begin{equation*}
\mathcal{Q}_1^{(m)}=\bigl(A_1^{(m)},B_1^{(m)},C_1^{(m)},k_1^{(m)},a_{k_1}^{(m)},D_1^{(m)}\bigr)
\end{equation*} 
constituted of polynomials $A_1^{(m)},B_1^{(m)},C_1^{(m)},D_1^{(m)}$, integers $k_1^{(m)}$ and numbers $a_{k_1}^{(m)}\in\{\pm 1\}$ defined  as follows:
\begin{enumerate}
\item If $m=1$: 
\begin{align*}
A_1^{(m)} &=-1,&
B_1^{(m)} &=-\frac{{\left(q^{2} - q + 1\right)} {\left(q^{n} + 1\right)} - 2 \, q}{q - 1},&
C_1^{(m)} &=q,\\
k_1^{(m)} &=0,&
a_{k_1}^{(m)} &=-1,&
D_1^{(m)} &=1-q.
\end{align*}

\item If $m=3j+1$ with $1≤j≤n-2$: 
\begin{align*}
A_1^{(m)} &=-\frac{{\left(q^{2} - q + 1\right)} {\left(q^{n-j} - q^{n} - 1\right)} + q^{j + 1}}{{\left(q - 1\right)}^{2}},&
B_1^{(m)} &=-\frac{{\left(q^{2} - q + 1\right)} {\left(q^{n} + 1\right)} - 2 \, q^{j + 1}}{q - 1},\\
C_1^{(m)} &=-q^{j + 1},&
k_1^{(m)} &=0,\\
a_{k_1}^{(m)} &=1,&
D_1^{(m)} &=1-q.
\end{align*}

\item If $m=3j+2$  with $0≤j≤n-2$: 
\begin{align*}
A_1^{(m)} &={\left(q^{2} - q + 1\right)} q^{n-j - 2},&
B_1^{(m)} &=\frac{{\left(q^{2} - q + 1\right)} {\left(2 \, q^{n-j} - q^{n} - 1\right)}}{q - 1},\\
C_1^{(m)} &=\frac{{\left(q^{4} - q^{3} + q^{2}\right)} {\left(q^{n-j} - q^{n} - 1\right)} + q^{j + 3}}{{\left(q - 1\right)}^{2}},&
k_1^{(m)} &=n-2-j,\\
a_{k_1}^{(m)} &=1,&
D_1^{(m)} &=\frac{q^{n-j} - 1}{q - 1}=[n-j]_q.
\end{align*}

\item If $m=3j+3$  with $0≤j≤n-2$: 
\begin{align*}
A_1^{(m)} &=q^{j},&
B_1^{(m)} &=\left(q^{2} - q + 1\right) [n]_q,&
C_1^{(m)} &=-{\left(q^{2} - q + 1\right)} q^{n-j},\\
k_1^{(m)} &=j,&
a_{k_1}^{(m)} &=1,&
D_1^{(m)} &=[j+2]_q-q.
\end{align*}
\end{enumerate}
\end{defi}

Using \eqref{FE-Phin} we see that  definition of $ \bigl(A_1^{(1)}, B_1^{(1)}, C_1^{(1)}\bigr)$ is such that
\begin{equation*}
A_1^{(1)}+ B_1^{(1)}\Phi_n+ C_1^{(1)}\Phi_n^2=0,
\end{equation*}
i.e. this first triple corresponds to the initialization of the process described in \cref{section-algo}. Now we prove that our algorithm \alg produces successively all the sequences defined above.

\begin{lem}\label{lemma:1}
Suppose that $1≤ m≤ 3n-3$. Then 
\begin{equation*}
\alg\bigl(A_1^{(m)}, B_1^{(m)}, C_1^{(m)}\bigr)=
\bigl(A_1^{(m+1)}, B_1^{(m+1)}, C_1^{(m+1)}, k_1^{(m)}, a_{k_1}^{(m)}, D_1^{(m)}\bigr).
\end{equation*}
\end{lem}

\begin{proof} Let $m$ be an integer such that $1≤ m≤ 3n-3$. To make our calculations easier to read, we just write $A$, $B$, $C$, $k$, $a_k$, $D$ instead of $A_1^{(m)}$, $B_1^{(m)}$, $C_1^{(m)}$, $k_1^{(m)}$, $a_{k_1}^{(m)}$, $D_1^{(m)}$  respectively. Let us apply the algorithm \alg for each case introduced in \cref{defi-A_1}.
\begin{enumerate}[wide]
\item Case  $m=1$.
\begin{description}[leftmargin=0pt]
\item [Step 1] Since $ A = -1$, \eqref{defak} gives $ k = 0$ and $ a_k  = -1$.
\item [Step 2] Since $ C = q $, we have  $c_1=1$ and \eqref{defD4} says that the polynomial $D$ is defined by
\begin{equation*}
q - \frac{{\left(q^{2} - q + 1\right)} {\left(q^{n} + 1\right)} - 2 \, q}{q - 1}
= 1- q^{n+1} - \frac{ q^{n} -q }{q - 1} =D + O(  q^{2 })
\end{equation*}
and that $\deg(D)≤1$. We see immediately that $ D = 1-q $.
\item [Step 3] Applying the definitions, we easily check that $ A^{*} =  A^{(2)}$, $B^{*} =  B^{(2)}$, $C^{*} =  C^{(2)}$.
\end{description}

\item Case  $m=3j+1$ with $1≤j≤n-2$.
\begin{description}[leftmargin=0pt]
\item [Step 1] Here,
\begin{equation*}
A = -\frac{{\left(q^{2} - q + 1\right)} {\left(q^{n-j} - q^{n} - 1\right)} + q^{j + 1}}{{\left(q - 1\right)}^{2}}
=1  + O(q)
\end{equation*}
so that $ k = 0$ and $ a_k  = 1$.

\item [Step 2] Since $ C = -q^{j + 1} $,  we have $c_1=0$ and by \eqref{defD4},
$D$ is the unique polynomial of degree $≤1$ such that
\begin{equation*}
\fr{B}{A}=D +O(q^{2 }).
\end{equation*}
But
\begin{equation*}
\fr{B}{A}=\frac{{(\left(q^{2} - q + 1) (q^{n} + 1) - 2 \, q^{j + 1}\right)} {(q - 1)}}{{(q^{2} - q + 1)} {(q^{n-j} - q^{n} - 1)} + q^{j + 1}}=\frac{-1+2q+O(q^2)}{-1+q+O(q^2)}
\end{equation*}
because $1≤j≤n-2$, hence $ D = 1-q$.  
\item [Step 3] We obtain easily $ A^{*} =  A^{(m+1)},\  B^{*} =  B^{(m+1)},\  C^{*} =  C^{(m+1)}$.
\end{description}

\item Case $m=3j+2$  with $0≤j≤n-2$.
\begin{description}[leftmargin=0pt]
\item [Step 1]
Since 
$ A = {\left(q^{2} - q + 1\right)} q^{n-j - 2} = q^{n-j - 2} + O(q^{n-j - 1})$
we have $ k = n-j - 2$ and $a_k  = 1$.
\item [Step 2]
Since
\begin{equation*}
C = \frac{{\left(q^{4} - q^{3} + q^{2}\right)} {\left(q^{n-j} - q^{n} - 1\right)} + q^{j + 3}}{{\left(q - 1\right)}^{2}},
\end{equation*}
we have $c_1=0$. By \eqref{defD4}, $D$ must satisfy $\deg(D)≤n-j-1$ and
\begin{equation*}
\frac{2 \, q^{n-j} - q^{n} - 1}{q - 1}=D + O(  q^{n-j }).
\end{equation*}
Splitting the left-hand side as $\frac{ q^{n-j} - q^{n}}{q - 1}+\frac{q^{n-j} - 1}{q - 1}$, we deduce that 
\begin{equation*}
D = \frac{q^{n-j} - 1}{q - 1}=[n-j]_q.
\end{equation*}
\item [Step 3] With the definitions, we check that 
$ A^{*} =  A^{(m+1)},\  B^{*} =  B^{(m+1)},\  C^{*} =  C^{(m+1)} $.
\end{description}
\item Case $m=3j+3$  with $0≤j≤n-2$. 
\begin{description}[leftmargin=0pt]
\item [Step 1]
In this case $ A = q^{j } $, so that
$ k = j$ and $a_k  = 1$.
\item [Step 2]  Since
$
C = -{\left(q^{2} - q + 1\right)} q^{n-j}
$,
we have  $c_1=0$. By \eqref{defD4}, $D$ is a polynomial of degree $≤j+1$ which satisfies
\begin{equation*}
\left(q^{2} - q + 1\right) [n]_q=D + O(  q^{j + 2 }),
\end{equation*}
that is to say
\begin{equation*}
1+q^2+q^3+\cdots+q^{n-1}+q^{n+1}=D + O(  q^{j + 2 }).
\end{equation*}
Condition $j+1≤n-1$ implies that
\begin{equation*}
D = 1+q^2+q^3+\cdots+q^{j+1}=-q+[j+2]_q.
\end{equation*}
\item [Step 3] As in the previous cases, we easily find that 
$ A^{*} =  A^{(m+1)},\  B^{*} =  B^{(m+1)},\  C^{*} =  C^{(m+1)} $,
where $m+1 =3j+1$ with $1≤j≤n-1$. 
\end{description}
\end{enumerate}
Thus the lemma is proved.
\end{proof}

\cref{lemma:1} gives us the list of the first $3n-2$ sextuplets $\mathcal{Q}_j$ in \eqref{family} that are needed to write the $H$-fraction of $\Phi_n$. However, a simple calculation shows that the result of 
$$\alg\bigl(A_1^{(3n-2)}, B_1^{(3n-2)}, C_1^{(3n-2)}\bigr)$$
is not expressible in terms of the  data that were defined \cref{defi-A_1}, and we have to define another family of data to proceed calculations.

\begin{defi}\label{defi-A_2}  
Let $\bigl(\mathcal{Q}_2^{(m)}\bigr)_{m≥1}$ denote the sequence of sextuplets
\begin{equation*}
\mathcal{Q}_2^{(m)}=\bigl(A_2^{(m)},B_2^{(m)},C_2^{(m)},k_2^{(m)},a_{k_2}^{(m)},D_2^{(m)}\bigr)
\end{equation*} 
constituted of polynomials $A_2^{(m)},B_2^{(m)},C_2^{(m)},D_2^{(m)}$, integers $k_2^{(m)}$ and numbers $a_{k_2}^{(m)}\in\{\pm 1\}$ defined  as follows:
\begin{enumerate}
\item If $m=1$: 
\begin{align*}
A_2^{(m)} &=1+q^2[n-1]_q,&
B_2^{(m)} &=\cron-2q^n,&
C_2^{(m)} &=-q^{n},\\
k_2^{(m)} &=0,&
a_{k_2}^{(m)} &=1,&
D_2^{(m)} &=1.
\end{align*}

\item If $m=4$: 
\begin{align*}
A_2^{(m)} &=q^{n - 1},&
B_2^{(m)} &=\cron,&
C_2^{(m)} &=-q^{n + 2},\\
k_2^{(m)} &=n - 1,&
a_{k_2}^{(m)} &=1,&
D_2^{(m)} &=\cron+q^{n+1}.
\end{align*}

\item If $m=3j+1$ with $2≤j≤n$: 
\begin{align*}
A_2^{(m)} &={\left(q^{2} - q + 1\right)} q^{j - 2},&
B_2^{(m)} &=(q^{2} - q + 1)[n]_q,&
C_2^{(m)} &=-q^{n-j+ 2},\\
k_2^{(m)} &=j - 2,&
a_{k_2}^{(m)} &=1,&
D_2^{(m)} &=[j]_q.
\end{align*}

\item If $m=2$: 
\begin{align*}
A_2^{(m)} &=-q^{n - 1},&
B_2^{(m)} &=\cron+2q^{n+1},&
C_2^{(m)} &=-q^2-q^4[n-1]_q,&\\
k_2^{(m)} &=n - 1,&
a_{k_2}^{(m)} &=-1,&
D_2^{(m)} &=\cron+q^{n+1}.
\end{align*}

\item If $m=5$: 
\begin{align*}
A_2^{(m)} &=-1-q^2[n-1]_q,&
B_2^{(m)} &=\cron+2q^{n+1},&
C_2^{(m)} &=-q^{n + 1},\\
k_2^{(m)} &=0,&
a_{k_2}^{(m)} &=-1,&
D_2^{(m)} &=1.
\end{align*}

\item If $m=3j+2$ with  $2≤j≤n-1$: 
\begin{align*}
A_2^{(m)} &=-\frac{{\left(q^{2} - q + 1\right)} {\left(q^{j} - q^{n} - 1\right)} + q^{n-j+ 1}}{{\left(q - 1\right)}^{2}},&
B_2^{(m)} &=\frac{{\left(q^{2} - q + 1\right)} {\left(2 \, q^{j} - q^{n} - 1\right)}}{q - 1},\\
C_2^{(m)} &=-{\left(q^{2} - q + 1\right)} q^{j},&
k_2^{(m)} &=0,\\
a_{k_2}^{(m)} &=1,&
D_2^{(m)} &=-q + 1.
\end{align*}

\item If $m=3$: 
\begin{align*}
A_2^{(m)} &=q^{n},&
B_2^{(m)} &=\cron,&
C_2^{(m)} &=-q^{n + 1},\\
k_2^{(m)} &=n,&
a_{k_2}^{(m)} &=1,&
D_2^{(m)} &=\cron.
\end{align*}

\item If $m=3j+3$  with $1≤j≤n-1$: 
\begin{align*}
A_2^{(m)} &=q^{n-j- 1},&
B_2^{(m)} &=-\frac{{\left(q^{2} - q + 1\right)} {\left(q^{n} + 1\right)} - 2 \, q^{n-j+ 1}}{q - 1},\\
C_2^{(m)} &=\frac{{\left({\left(q^{2} - q + 1\right)} {\left(q^{j} - q^{n} - 1\right)} + q^{n-j+ 1}\right)} q^{2}}{{\left(q - 1\right)}^{2}},&
k_2^{(m)} &=n-1-j,\\
a_{k_2}^{(m)} &=1,&
D_2^{(m)} &=[n-j+1]_q-q.
\end{align*}
\end{enumerate}
\end{defi}

Notice that 
\begin{equation}\label{jonction}
\bigl(A_2^{(1)}, B_2^{(1)}, C_2^{(1)}\bigr)=
\bigl(A_1^{(3n-2)}, B_1^{(3n-2)}, C_1^{(3n-2)}\bigr).
\end{equation}
Thereby the second family $\bigl(\mathcal{Q}_2^{(m)}\bigr)_{m≥1}$ is  the natural continuation of the first one $\bigl(\mathcal{Q}_1^{(m)}\bigr)_{m=1}^{3n-2}$.

\begin{lem}\label{lemma:2}
Suppose that $1≤ m≤ 3n+1$. Then 
\begin{equation*}
\alg\bigl(A_2^{(m)}, B_2^{(m)}, C_2^{(m)}\bigr)=
\bigl(A_2^{(m+1)}, B_2^{(m+1)}, C_2^{(m+1)}, k_2^{(m)}, a_{k_2}^{(m)}, D_2^{(m)}\bigr).
\end{equation*}
\end{lem}

\begin{proof} As for  \cref{lemma:1}, the proof consists in applying the algorithm \alg for each case introduced in \cref{defi-A_2}. Calculations as exactly of the same kind, lengthy but not difficult, and we omit details here.
\end{proof} 

Notice that \cref{lemma:2} is not valid any more for next index $m=3n+2$, i.e.
$$\alg\bigl(A_2^{(3n+2)}, B_2^{(3n+2)}, C_2^{(3n+2)}\bigr)$$
will not give a sextuplet consisting of elements defined in \cref{defi-A_2}. Nevertheless, we make the following crucial observation: 
\begin{align*}
A_2^{(3n+1)} &=(q^{2} - q + 1) q^{n - 2}=A_1^{(2)},\\
B_2^{(3n+1)} &=(q^{2} - q + 1)[n]_q=B_1^{(2)},\\
C_2^{(3n+1)} &=-q^{2}=C_1^{(2)},\\
k_2^{(3n+1)} &=n - 2=k_1^{(2)},\\
a_{k_2}^{(3n+1)} &=1=a_{k_1}^{(2)},\\
D_2^{(3n+1)} &=[n]_q=D_1^{(2)},
\end{align*}
and therefore
\begin{equation*}
\alg\bigl(A_2^{(3n+1)}, B_2^{(3n+1)}, C_2^{(3n+1)}\bigr)
=\bigl(A_1^{(3)}, B_1^{(3)}, C_1^{(3)}, k_1^{(2)}, a_{k_1}^{(2)}, D_1^{(2)}\bigr).
\end{equation*}
In other words we came back to the first family $\bigl(\mathcal{Q}_1^{(m)}\bigr)$ of sextuplets and \emph{this proves the periodicity of the $H$-fraction expansion of $\Phi_n$}. Indeed, we just proved that the family \eqref{family} $(\mathcal{Q}_j)_{j≥0}$ of sextuplets associated with $\Phi_n$ is completely determined by the following equalities:
\begin{align}
&\text{first terms:}\quad\mathcal{Q}_j=
\begin{dcases*}
\mathcal{Q}_1^{(j+1)}&if $0≤j≤3n-4$,\\
\mathcal{Q}_2^{(j-3n+4)}&if $3n-3≤j≤6n-4,$
\end{dcases*}\label{defQj}\\
&\text{periodicity:}\quad\mathcal{Q}_{j+6n-4}=\mathcal{Q}_j\quad\text{for }j≥1.\notag
\end{align}

We summarize the results of this section in \cref{table-alg} below. Extracting the data sequence $(k_j,a_j,D_j)$ from that table and applying \eqref{Hfracalgo} (with $p=1$ and $p'=6n-3$) gives exactly \eqref{Hfrac-Phin}, hence \cref{MainHF} is proved. In the table, the first line $j=0$ corresponds to the head of  the $H$-fraction (first term) and the periodic part consists of the block of lines from $j=1$ to $j=6n-4$.
\medbreak
\begin{table}[t]
\caption{Results of the algorithm \alg applied to $\Phi_n$.}\label{table-alg}
{%
\newcolumntype{C}{>{$}c<{$}} %
\renewcommand{\arraystretch}{1.4}
\begin{center}
\begin{tabular}{ |C|C|C|C|C|C|} 
\hline
j & \mathcal{Q}_j & k_j& a_j & D_j&\text{Contribution} \\
\hline
0 & \mathcal{Q}_1^{(1)} & 0& -1 & 1-q &\text{Head}\\
\hline
1 & \mathcal{Q}_1^{(2)} & n - 2& 1 & [n]_q& U_n(q)  \\
2 & \mathcal{Q}_1^{(3)} & 0& 1 & 1& \\
3 & \mathcal{Q}_1^{(4)} & 0& 1 & 1-q &\\
\vdots & \vdots  &\vdots & \vdots &\vdots &\\
3n-5 & \mathcal{Q}_1^{(3n-4)}& 0& 1 & 1+q& \\
3n - 4 & \mathcal{Q}_1^{(3n-3)}  & n - 2& 1 & [n]_q-q& \\
3n-3 & \mathcal{Q}_1^{(3n-2)}=\mathcal{Q}_2^{(1)} & 0& 1 & 1 &\\
\hline
3n-2 & \mathcal{Q}_2^{(2)} & n - 1& -1 & \cron+q^{n+1} & V_n(q) \\
3n-1 & \mathcal{Q}_2^{(3)} & n& 1 & \cron& \\
3n & \mathcal{Q}_2^{(4)} & n - 1& 1 & \cron+q^{n+1}& \\
3n+1 & \mathcal{Q}_2^{(5)} & 0& -1 & 1 &\\
\hline
3n+2 & \mathcal{Q}_2^{(6)} & n - 2& 1 & [n]_q-q & W_n(q)\\
3n+3 & \mathcal{Q}_2^{(7)} & 0& 1 & 1+q & \\
3n+4 & \mathcal{Q}_2^{(8)} & 0& 1 & 1-q & \\
\vdots & \vdots  &\vdots & \vdots & \vdots &\\
6n-7 & \mathcal{Q}_2^{(3n-3)} & 1& 1 & 1+q^{2}& \\
6n-6 & \mathcal{Q}_2^{(3n-2)} & n - 3& 1 & [n-1]_q& \\
6 n-5 & \mathcal{Q}_2^{(3n-1)} & 0& 1 & 1-q & \\
6n-4 & \mathcal{Q}_2^{(3n)} & 0& 1 & 1 &\\
\hline
6n-3 & \mathcal{Q}_2^{(3n+1)}=\mathcal{Q}_1^{(2)} & n - 2& 1 & [n]_q &\text{2nd period} \\
\vdots & \vdots  &\vdots &\vdots  &\vdots & \\
\hline
\end{tabular}
\end{center}
}
\end{table}

\subsection{A remark on symmetries}\label{section-symmetries}

Formula \eqref{Hfrac-Phin} reveals three types of symmetries \emph{inside} the period of the continued fraction. Let $(\al_i)$ and $(\be_i)$ denote the sequences of numerators and denominators of the continued fraction \eqref{Hfrac-Phin}, respectively, i.e. write
\begin{equation*}
\Phi_n(q)=\CF_{i=0}^{+\infty}\fr{\al_i}{\be_i}=\fr{\al_0}{\be_0}\cfp\left(\CF_{i=1}^{6n-4}\fr{\al_i}{\be_i}\cfp\right)^*.
\end{equation*}
\begin{enumerate}
\item Symmetries inside the whole period (and, actually, a bit larger block):
\begin{align*}
&\al_i=\al_{6n-1-i},\quad i=1,2,\ldots,6n-2.\\
&\be_i=\be_{6n-2-i},\quad i=1,2,\ldots,6n-3.
\end{align*}

\item Symmetries inside the first block $U_n(q)$:
\begin{align*}
&\al_i=\al_{3n-2-i},\quad i=1,2,\ldots,3n-3.\\
&\be_0=\be_{3n-3}-q\quad \text{and}\quad \be_i=\be_{3n-3-i}+\chi_i q,\quad i=1,2,\ldots,3n-4,\\
&\qquad\text{ where }\chi_i=
\begin{dcases*}
0&if $i\equiv 0\mod {3}$\\
1&if $i\equiv 1\mod {3}$\\
-1&if $i\equiv -1\mod {3}$.
\end{dcases*}
\end{align*}

Symmetries of type (1) and (2) imply also similar symmetries inside the block $W_n(q)$.
\item `Pseudo half-period' for numerators inside the whole period:
\begin{equation*}
\al_i=\al_{i+3n+1},\quad i=1,\ldots,3n-3.
\end{equation*}

These identities reflect  that the sequence of numerators in block $W_n(q)$ is essentially the same as in block $U_n(q)$.
\end{enumerate}

\begin{ex}\label{ex-Phi5} Let us write the $H$-fraction and illustrate its symmetries in the case $n=5$.
We have
\begin{equation}\label{Hfrac-Phi5}
\Phi_5(q)=\frac{1}{-q+1}\cfp\left(U_5(q)\cfp V_5(q)\cfp W_5(q)\cfp\right)^*
\end{equation}
with
\begin{align*}
	U_5(q)=\CF_{i=1}^{12}\fr{\al_i}{\be_i}=&
	\fr{q^5}{ [5]_q }\cfp
\fr{q^5}{1}\cfp
\fr{q^2}{1-q }\cfp
	\fr{q^4}{[4]_q}\cfp
\fr{q^5}{q^2 + 1}\cfp
\fr{q^3}{1-q }\cfp
\\
	&\quad\fr{q^3}{[3]_q }\cfp
\fr{q^5}{ [4]_q - q   }\cfp
\fr{q^4}{1-q }\cfp
	\fr{q^2}{[2]_q}\cfp
\fr{q^5}{ [5]_q - q }\cfp
\fr{q^5}{1},
\end{align*}
\begin{equation*}
	V_5(q)=\CF_{i=13}^{16}\fr{\al_i}{\be_i}=
\fr{-q^6}{  \langle 5 \rangle_q  +q^6 }\cfp
\fr{q^{11}}{ \langle 5 \rangle_q }\cfp
\fr{q^{11}}{  \langle 5 \rangle_q  +q^6  }\cfp
\fr{-q^6}{1}, 
\end{equation*}
and
\begin{align*}
	W_5(q)=\CF_{i=17}^{26}\fr{\al_i}{\be_i}=&
	\fr{q^5}{[5]_q - q }\cfp
	\fr{q^5}{[2]_q }\cfp
\fr{q^2}{1-q }\cfp
\fr{q^4}{ [4]_q - q  }\cfp
	\fr{q^5}{[3]_q }\cfp
\fr{q^3}{1-q }\cfp
\\
&\quad\fr{q^3}{q^2 + 1}\cfp
	\fr{q^5}{[4]_q }\cfp
\fr{q^4}{1-q }\cfp
\fr{q^2}{1}. 
\end{align*}
We can check  symmetries of type (1):
\begin{align*}
&\al_i=\al_{29-i},\quad i=1,2,\ldots,28.\\
&\be_i=\be_{28-i},\quad i=1,2,\ldots,27.
\end{align*}
as well as  symmetries of type (2):
\begin{align*}
&\al_i=\al_{13-i},\quad i=1,\ldots,12.\\
&\be_0=\be_{12}-q\quad \text{and}\quad \be_i=\be_{12-i}+\chi_i q,\quad i=1,\ldots,11,\\
\end{align*}
and of type (3):
\begin{equation*}
\al_i=\al_{i+16},\quad i=1,\ldots,12.
\end{equation*}
\end{ex}

\section{Special values and (anti-)periodicity property}\label{section-period}

As mentioned in the introduction, the statement of \cref{MainValPer} for $1≤\ell≤n+1$ obviously follows from \cref{MainContiguity} and the case $\ell=0$. This section is devoted to the proof of the latter and we still assume that $n≥3$.

\subsection{Collecting some data}

In this subsection we give the explicit values of  the family of parameters $(s_j)$ and $(\eps_j)$ introduced in \eqref{s-def} and \eqref{e-def}, because they are needed to compute the Hankel determinants $\De_j(\Phi_n)$ with formula \eqref{HankelHF}. In particular, we will emphasize on some symmetric patterns they form, because these will simplify further calculations.  

First, we know from the results in \cref{section-pfD} that the sequence $(k_i)$ is $(6n-4)$-periodic and determined by the following first values: 
\begin{equation}\label{value-kj}
\left\{
\begin{array}{llll}
k_0=0, &&&
\\
k_{3i+1}=n-2-i,&
k_{3i+2}=i,&
k_{3i+3}=0,&
\text{for }i=0,\ldots,n-2
\\
k_{3n-2}=n-1,&
k_{3n-1}=n,&
k_{3n}=n-1,&
\\
k_{3i+1}=0,&
k_{3i+2}=2n-2-i,&
k_{3i+3}=i-n,&
\text{for }i=n,\ldots,2n-3
\\
k_{6n-5}=0.&&&
\end{array}
\right.
\end{equation}
Hence we have:

\begin{lem} The first period of the sequence $(k_i)$ admits the following symmetries:
\begin{enumerate}
\item Symmetry inside the whole period $\{k_0,\ldots,k_{6n-3}\}$:
\begin{equation}
\label{symper-k}
k_i=k_{6n-2-i}\quad\text{for }i=0,\ldots,6n-2.
\end{equation}
\item One-to-one correspondence between subsets $\{k_0,\ldots,k_{3n-3}\}$ and $\{k_{3n+1},\ldots,k_{6n-2}\}$ of `non-central' values:
\begin{equation}
\label{bij-k}
k_{i+3n+1}=k_i\quad\text{for }i=0,\ldots,3n-3.
\end{equation}
\item Symmetries inside the subsets $\{k_0,\ldots,k_{3n-3}\}$ and $\{k_{3n+1},\ldots,k_{6n-2}\}$:
\begin{align}
&k_i=k_{3n-3-i}\quad\text{for }i=0,\ldots,3n-3\label{symdemiper-k}\\
&k_{3n+1+i}=k_{6n-2-i}\quad\text{for }i=0,\ldots,3n-3.\notag
\end{align}
\end{enumerate}
\end{lem}

\begin{ex} For a better understanding, let us illustrate the symmetries when $n=5$. Here $3n-3=12$, $3n+1=16$, $6n-2=28$ and the $k_i$'s are given by the array:
\begin{align*}
&\begin{array}{l|*{13}{wc{1em}}}
i & 0 & 1 & 2 & 3 & 4 & 5 & 6 & 7 & 8 & 9 & 10 & 11 & 12\\ 
\hline
k_i & 0 & 3 & 0 & 0 & 2 & 1 & 0 & 1 & 2 & 0 & 0 & 3 & 0 \\ 
\end{array}
\\
&\begin{array}{l|*{3}{wc{1em}}}
i & 13 & 14 & 15 \\ 
\hline
k_i & 4 & 5 & 4 \\ 
\end{array}
\\
&\begin{array}{l|*{13}{wc{1em}}}
i & 16 & 17 & 18 & 19 & 20 & 21 & 22 & 23 & 24 & 25 & 26 & 27 & 28\\ 
\hline
k_i & 0 & 3 & 0 & 0 & 2 & 1 & 0 & 1 & 2 & 0 & 0 & 3 & 0 \\ 
\end{array}
\end{align*}
Thus we have 
\begin{align*}
	k_i&=k_{28-i},\quad i=0,\ldots,14; &\qquad k_{i+16}&=k_{i},\quad i=0,\ldots,12, \\
	k_i&=k_{12-i},\quad i=0,\ldots,6; &\qquad k_{16+i}&=k_{28-i},\quad i=0,\ldots,6.
\end{align*}
\end{ex}

Recall from \eqref{HankelHF} that the numbers $s_j\in\cS$ enumerate the nonzero Hankel determinants of $\Phi_n$. We determine now their values:
\begin{lem} \label{value-sj}
\begin{enumerate}
\item The values of the parameter $s_j$, for $0≤j≤6n-4$, are given by
\begin{align*}
s_{3p}&=
\begin{dcases*}
p(n+1),& if $0≤p≤n-1$,\\
(n+1)^2,& if $p=n$,\\
1+(p+3)n,& if $n+1≤p≤2n-2$;
\end{dcases*}
\\
s_{3p+1}&=
\begin{dcases*}
1+p(n+1),& if $0≤p≤n-1$,\\
n+(p+1)(n+1),& if $n≤p≤2n-2$;
\end{dcases*}
\\
s_{3p+2}&=
\begin{dcases*}
(p+1)n,& if $0≤p≤n-2$,\\
n(n+1),& if $p=n-1$,\\
(p+2)(n+1),& if $n≤p≤2n-2$.
\end{dcases*}
\end{align*}
\item We have the following relations:
\begin{align}
&s_{j+3n+1}=s_j+s_{3n+1}=s_j+n+(n+1)^2\quad\text{for }0≤j≤3n-2,\label{sym1-s}\\
&s_j+s_{6n-1-j}=s_{6n-1}=(2n+1)(n+1)\quad\text{for }0≤j≤6n-1.\label{sym2-s}
\end{align}
\end{enumerate}
\end{lem}

\begin{proof} Recall that, by definition, $s_0=0$ and $s_j=j+\sum_{i=0}^{j-1}k_i$. Thus, from the values \eqref{value-kj} of the $k_i$'s it is  easy  to compute first all values of $s_j$ for $0≤j≤3n+1$. Then we prove \eqref{sym1-s} by using \eqref{bij-k}:
\begin{equation*}
s_{j+3n+1}=s_{3n+1}+j+\sum_{i=3n+1}^{3n+j}k_i=s_{3n+1}+j+\sum_{i=0}^{j-1}k_i=s_{3n+1}+s_j
\end{equation*}
and in particular these symmetries give the remaining values of $s_j$ for $3n+2≤j≤6n-4$. Similarly, the second symmetry $s_j+s_{6n-1-j}=s_{6n-1}$ follows from \eqref{symper-k} and it remains only to compute, with the help of \eqref{sym1-s}
\begin{equation*}
s_{6n-1}=s_{3n-2}+s_{3n+1}=n+(n+1)^2+n^2=(2n+1)(n+1).\qedhere
\end{equation*}
\end{proof}

\begin{lem} \label{value-epsj}
\begin{enumerate}
\item The values of the parameter $\eps_j$, for $0≤j≤6n-4$, are given by
\begin{align*}
\eps_{3p}&=
\begin{dcases*}
\textstyle\fr{p(n-p-1)(n-1)}{2}+\fr{p(p-1)(p+1)}{3},& if $0≤p≤n-1$,\\
\textstyle\fr{n(n^2+2)}{3},& if $p=n$,\\
\textstyle\fr{n(n+1)(2n+1)}{6}+\fr{(p-n)(2n-p)(n-2)}{2}+\fr{(p-n)(p-n-1)(p-n-2)}{3},& if $n+1≤p≤2n-2$;
\end{dcases*}
\\
\eps_{3p+1}&=
\begin{dcases*}
\eps_{3p},& if $0≤p≤n-1$,\\
\textstyle \frac{n(n+1)(2n+1)}{6}+\fr{(p-n)(2n-p-1)(n-1)}{2}+\fr{(p-n)(p-n-1)(p-n+1)}{3},& if $n≤p≤2n-2$;
\end{dcases*}
\\
\eps_{3p+2}&=
\begin{dcases*}
\textstyle\fr{(p+1)(n-p-1)(n-2)}{2}+\fr{p(p-1)(p+1)}{3},& if $0≤p≤n-2$,\\
\textstyle\fr{n(n-1)(2n-1)}{6},& if $p=n-1$,\\
\eps_{3p+1},& if $n≤p≤2n-2$.
\end{dcases*}
\end{align*}
\item We have the following relations:
\begin{align}
&\textstyle\eps_{j+3n+1}=\eps_j+\eps_{3n+1}=\eps_j+\fr{n(n+1)(2n+1)}{6}\quad\text{for }0≤j≤3n-2,\label{sym1-eps}\\
&\textstyle\eps_j+\eps_{6n-1-j}=\eps_{6n-1}=\fr{n(n+1)(2n+1)}{6}+\fr{n(n-1)(n-2)}{3}\quad\text{for }0≤j≤6n-1.\label{sym2-eps}
\end{align}
\end{enumerate}
\end{lem}

\begin{proof}
Recall that $\eps_j:=\sum_{i=0}^{j-1}\frac{k_i(k_i+1)}{2}$ with  convention $\eps_0:=0$. As one can guess, computations rely essentially on the values \eqref{value-kj} of the $k_j$'s and on the classical formula
\begin{equation*}
\sum_{i=a}^{a+b-1}i(i+1)=ab(a+b)+\fr{b(b-1)(b+1)}3
\end{equation*}
for nonnegative integers $a$ and $b$, at least for $\eps_j$ with $j≤3n+1$. The remaining cases, as in the previous lemma, will follow from the symmetry relations \eqref{sym1-eps} and \eqref{sym2-eps}. To prove the first one, we simply use \eqref{bij-k}, while to prove the second one we use \eqref{sym1-s}: we write for $1≤j≤6n-1$ (the remaining case $j=0$ is trivial)
\begin{equation*}
\eps_j=\sum_{i=0}^{j-1}\frac{k_i(k_i+1)}{2}
=\sum_{i=0}^{j-1}\frac{k_{6n-2-i}(k_{6n-2-i}+1)}{2}
=\sum_{i=6n-1-j}^{6n-2}\frac{k_{i}(k_{i}+1)}{2}
\end{equation*}
so that
\begin{equation*}
\eps_j+\eps_{6n-1-j}=\sum_{i=6n-1-j}^{6n-2}\frac{k_{i}(k_{i}+1)}{2}
+\sum_{i=0}^{6n-2-j}\frac{k_{i}(k_{i}+1)}{2}=\eps_{6n-1}.
\end{equation*}
We then employ \eqref{sym1-eps} to get that
\begin{equation*}
\eps_{6n-1}=\eps_{3n+1}+\eps_{3n-2}=\fr{n(n+1)(2n+1)}{6}+\fr{n(n-1)(n-2)}{2}.
\qedhere
\end{equation*}
\end{proof}

\subsection{Proof of \texorpdfstring{\cref{MainValPer}}{Theorem~\ref{MainValPer}} in the case $\ell=0$}

First we note that, when $\ell=0$, first assertion in \cref{MainValPer} follows immediately from \cref{MainHF}. Indeed, we see that  the sequence $(v_i)$ defined in \eqref{superfrac} can only have values $\pm1$. More precisely, using the fact that the sequence is $(6n-4)$-periodic with offset $1$, i.e. $v_{i+6n-4}=v_i$ for all $i≥1$ and by looking at the list $\{v_0,v_1,\ldots,v_{6n-4}\}$ in \eqref{Hfrac-Phin} we get that
\begin{equation}\label{value-vj}
v_j=
\begin{dcases*}
1& if $j=0$, or $j\equiv 3n-2\mod {6n-4}$, or $j\equiv 3n+1\mod {6n-4}$,\\
-1& else.
\end{dcases*}
\end{equation}
Applying formula \eqref{HankelHF} we deduce that Hankel determinants $\De_{j}$ can only take the values $-1,0,1$.

Thus, it remains to prove the second assertion in \cref{MainValPer}, i.e. formula
\begin{equation}\label{per-De0}
\De_{j+2n(n+1)}=(-1)^n\De_{j}\quad\text{for all } j≥0.
\end{equation}
In fact, to do this, it is not even necessary to know the explicit expression of the Hankel determinants, which will be given in \cref{section-HD0}. 

\begin{lem}\label{sp+6n-4}
We have 
$s_{p+6n-4}=s_p+2n(n+1)$
for all $p≥0$.
\end{lem}

\begin{proof}
By definition,
\begin{equation*}
s_{p+6n-4}=p+6n-4+\sum_{i=0}^{p-1}k_i=s_{6n-4}+p+\sum_{i=6n-4}^{6n-4+p}k_i.
\end{equation*} 
Using the $(6n-4)$-periodicity of the sequence $(k_i)$ we obtain that $s_{p+6n-4}=s_{6n-4}+s_p$, and we  read the value 
\begin{equation}\label{s-6n-4}
s_{6n-4}=2n(n+1)
\end{equation}
from Lemma~\ref{value-sj}.
\end{proof}

From this lemma we deduce that $j\in \cS$ if and only if $j+2n(n+1)\in \cS$ and that it suffices to prove \eqref{per-De0} for $j\in\cS$ (see \eqref{s-def} for the definition of the set $\cS$). In other words, we are left to prove that
\begin{equation*}
\De_{s_p+2n(n+1)}=(-1)^n\De_{s_p}\quad\text{for all }p≥0.
\end{equation*}
This will be an immediate consequence of Lemma~\ref{sp+6n-4} and of the following result.

\begin{prop} \label{Desp+6n-4}
We have
$\De_{s_{p+6n-4}}=(-1)^n\De_{s_p}$
for all $p≥0$.
\end{prop}

\begin{proof}
First of all, note that, since  parameters $v_i$ can only be $\pm1$ 
we can rewrite the expression of nonzero Hankel determinants \eqref{HankelHF} as follows:
\begin{equation}\label{Hankel-prod}
\Delta_{s_p}
=(-1)^{\eps_p}(v_0 v_1\cdots v_{p-1})^{s_p}\,v_1^{s_1}v_2^{s_2}\cdots{}v_{p-1}^{s_{p-1}}.
\end{equation}
However, in this formula we must exclude the case $p=0$, i.e. $s_p=0$, where $\De_0=1$ by convention, and this leads us to divide our proof into two parts.
\smallbreak

$\bullet$ Let us first suppose that $p≥1$. In this case, formula \eqref{Hankel-prod} also gives
\begin{align*}
\Delta_{s_{p+6n-4}}
&=(-1)^{\eps_{p+6n-4}}\bigl(v_0 v_1\cdots v_{p+6n-5})^{s_{p+6n-4}}\,v_1^{s_1}v_2^{s_2}\cdots v_{p+6n-5}^{s_{p+6n-5}}\\
&=(-1)^{\eps_{p+6n-4}}\bigl(v_0 v_1\cdots v_{6n-4}\bigr)^{s_{p+6n-4}}
\bigl(v_{6n-3} v_{6n-2} \cdots v_{p+6n-5}\bigr)^{s_{p+6n-4}}\\
&\qquad\times\bigl(v_1^{s_1}v_2^{s_2}\cdots v_{6n-4}^{s_{6n-4}}\bigr)
 \bigl(v_{6n-3}^{s_{6n-3}}v_{6n-2}^{s_{6n-2}}\cdots v_{p+6n-5}^{s_{p+6n-5}}\bigr).
\end{align*}
Let us simplify  last expression. Firstly, because of \eqref{value-vj} we have 
\begin{equation*}
v_0 v_1\cdots v_{6n-4}=(-1)^{6n-3-3}=1.
\end{equation*} 
Secondly, by using the $(6n-4)$-periodicity of the sequence $(k_i)$ we see that $\eps_{p+6n-4}=\eps_{6n-4}+\eps_p$, with 
\begin{equation}\label{eps-6n-4}
\eps_{6n-4}=\fr{(2n-1)(n^2-n+3)}3
\end{equation}
 by Lemma~\ref{value-epsj} and we note that this is always an odd number. Thirdly, since the sequence $(v_i)$ is $(6n-4)$-periodic with values in $\{\pm 1\}$, since $s_{p+6n-4}$ has same parity as $s_{p}$ by Lemma~\ref{sp+6n-4} and since $v_0=1$, we get
\begin{equation*}
\bigl(v_{6n-3} v_{6n-2} \cdots v_{p+6n-5}\bigr)^{s_{p+6n-4}}
=\bigl(v_{1} v_{2} \cdots v_{p-1}\bigr)^{s_{p}}
=\bigl(v_0 v_{1} \cdots v_{p-1}\bigr)^{s_{p}}
\end{equation*}
and, similarly,
\begin{equation*}
v_{6n-3}^{s_{6n-3}}v_{6n-2}^{s_{6n-2}}\cdots v_{p+6n-5}^{s_{p+6n-5}}
=v_{1}^{s_{1}}v_{2}^{s_{2}}\cdots v_{p-1}^{s_{p-1}}.
\end{equation*}
To sum up, so far we have proved that
\begin{align*}
\Delta_{s_{p+6n-4}}
&=(-1)^{\eps_{p}+1}
\bigl(v_{0} v_{1} \cdots v_{p-1}\bigr)^{s_{p}}
\bigl(v_{1}^{s_{1}}v_{2}^{s_{2}}\cdots v_{p-1}^{s_{p-1}}\bigr)
\bigl(v_1^{s_1}v_2^{s_2}\cdots v_{6n-4}^{s_{6n-4}}\bigr)\\
&=-\De_{s_p}\bigl(v_1^{s_1}v_2^{s_2}\cdots v_{6n-4}^{s_{6n-4}}\bigr).
\end{align*}
It remains to calculate the product in the right-hand side of this equality, which by \eqref{value-vj} equals
\begin{equation*}
(-1)^{s_1+s_2+\cdots+s_{6n-4}-s_{3n-2}-s_{3n+1}}.
\end{equation*}
But
\begin{multline*}
s_1+s_2+\cdots+s_{6n-4}-s_{3n-2}-s_{3n+1}=s_1+s_2+\cdots+s_{6n-2}-(s_{6n-3}+s_{6n-2})-(s_{3n-2}+s_{3n+1})
\end{multline*}
with
\begin{align}
s_{6n-3}&=s_{3n-4}+s_{3n+1}=2n^2+2n+1,\label{s6n-3}\\
s_{6n-2}&=s_{3n-3}+s_{3n+1}=2n^2+3n\label{s6n-2}
\end{align}
by Lemma~\ref{value-sj}, and
\begin{equation*}
s_1+s_2+\cdots+s_{6n-2}-(s_{3n-2}+s_{3n+1})=(3n-2)(2n+1)(n+1)
\end{equation*}
by \eqref{sym2-s}, so that
\begin{align*}
s_1+s_2+\cdots+s_{6n-4}-s_{3n-2}-s_{3n+1}
&=(3n-2)(2n+1)(n+1)-(4n+1)(n+1)\\
&=(n+1)(6n^2-5n-3).
\end{align*}
Finally we find that
\begin{equation}\label{vi-prod}
v_1^{s_1}v_2^{s_2}\cdots v_{6n-4}^{s_{6n-4}}=(-1)^{s_1+s_2+\cdots+s_{6n-4}-s_{3n-2}-s_{3n+1}}=(-1)^{n+1}
\end{equation}
and
\begin{equation*}
\De_{s_{p+6n-4}}=(-1)^n\De_{s_p}\quad\text{for all }p≥1
\end{equation*}
as required.
\smallbreak

$\bullet$ To finish our proof, it remains to treat the case $p=0$. Namely, we have  to check that $\De_{s_{6n-4}}=(-1)^n$, where
\begin{equation*}
\Delta_{s_{6n-4}}
=(-1)^{\eps_{6n-4}}(v_0 v_1\cdots v_{6n-5})^{s_{6n-4}}\,v_1^{s_1}v_2^{s_2}\cdots{}v_{6n-5}^{s_{6n-5}}
\end{equation*}
by \eqref{Hankel-prod}.
But the integer $\eps_{6n-4}$ is odd by \eqref{eps-6n-4} and the integer $s_{6n-4}$ is even by \eqref{s-6n-4}. Together with the result of \eqref{vi-prod} this yields
\begin{equation*}
\Delta_{s_{6n-4}}=(-1)\cdot 1\cdot (-1)^{n+1}=(-1)^n.\qedhere
\end{equation*}
\end{proof}

This completes the proof of \cref{MainValPer} in the case $\ell=0$.

\section{\texorpdfstring{The Hankel determinants $\De_j(\Phi_n)$}{The Hankel determinants}}
\label{section-HD0}

Recall from \eqref{HankelHF} that the nonzero Hankel determinants $\De_j$ are exactly the ones for which $j\in\cS$, i.e. $j=s_q$ for some integer $q≥0$. We give now their explicit values, taking into account that it suffices to restrict ourselves to a finite subsequence of length $2n(n+1)$ which is the period when $n$ is even (resp. the anti-period when $n$ is odd), by the results proved in the previous section.

\begin{thm} \label{thm-Hankel0}
\begin{enumerate}
\item The only nonzero Hankel determinants in the list $\De_0,\De_1,\ldots,\De_{2n(n+1)-1}$ formed by the first (anti-)period are the $\De_{s_q}$ with $0≤q≤6n-5$. Their explicit values are as  follows:
\begin{align}
\De_{s_{3p}}&=
\begin{dcases*}
(-1)^{p\fr{n(n-1)}{2}}(-1)^{\fr{p(p-1)}2},& if $0≤p≤n-1$,\\
(-1)^{\fr{n(n+1)(n+2)}{2}},& if $p=n$,\\
(-1)^{(p+1)\fr{n(n-1)}{2}}(-1)^n,& if $n+1≤p≤2n-2$;
\end{dcases*}
\label{Des3p}\\ 
\De_{s_{3p+1}}&=(-1)^{p\fr{n(n-1)}{2}}(-1)^{\fr{p(p+1)}2}, \qquad\text{if }0≤p≤2n-2;
\label{Des3p+1}\\ 
\De_{s_{3p+2}}&=
\begin{dcases*}
(-1)^{(p+1)\fr{n(n-1)}{2}},& if $0≤p≤n-2$,\\
(-1)^{\fr{n(n-1)^2}2},& if $p=n-1$,\\
(-1)^{p\fr{n(n-1)}{2}}(-1)^{\fr{(p+1)(p+2)}2},& if $n≤p≤2n-3$.
\end{dcases*}
\label{Des3p+2}
\end{align}
\item Besides the (anti-)periodicity \eqref{per-De0} already proved, the Hankel determinants satisfy the following symmetry property:
\begin{equation}\label{symetrie-De0}
\De_{j}=(-1)^{\fr{n(n+1)}{2}}\De_{(2n+1)(n+1)-j}\quad \text{for }0≤j≤(2n+1)(n+1).
\end{equation}
\end{enumerate} 
\end{thm}
It is important to note that, since $(2n+1)(n+1)$ is strictly greater than the length of the (anti-)period $2n(n+1)$,  symmetry \eqref{symetrie-De0} holds inside the whole first (anti-)period.

\begin{ex} Let us give the explicit values of Hankel determinants $\De_j=\De_j(\Phi_n)$ when $1≤n≤5$.
\begin{enumerate}
\item Recall from \eqref{HankelGold2} that sequence $\De(\Phi_1)$ is $4$-antiperiodic and starts with
\begin{equation*}
1,1,1,0.
\end{equation*}
\item Sequence $\De(\Phi_2)$ is $12$-periodic and the period is
\begin{equation*}
1,\,1,-1,-1,\,1,\,0,\,-1,\,0,\,0,\,1,\,0,-1.
\end{equation*}
(See also Theorem~1.8 in \cite{OP25}.)
\item Sequence $\De(\Phi_3)$ is $24$-antiperiodic and starts with
\begin{equation*}
1,1,0,-1,-1,1,1,0,-1,-1,0,0,1,0,0,0,1,0,0,-1,-1,0,1,1.
\end{equation*}
\item Sequence $\De(\Phi_4)$ is $40$-periodic and the period is
\begin{align*}
&1, 1, 0, 0, 1, 1, -1, 0, 1, 0, -1, -1, 1, 0, 0, -1, 1, 0, 0, 0,\\
& 1, 0, 0, 0, 0, 1, 0, 0, 0, 1, -1, 0, 0, 1, -1, -1, 0, 1, 0, -1.
\end{align*}
\item Sequence $\De(\Phi_5)$ has an antiperiodicity of length $60$ and starts with
\begin{align*}
&1, 1, 0, 0, 0, 1, 1, -1, 0, 0, 1, 0, -1, -1, 0, 1, 0, 0, -1, 1, 1, 0, 0, 0, 1, 1, 0, 0, 0, 0,\\
& 1, 0, 0, 0, 0, 0, -1, 0, 0, 0, 0, -1, -1, 0, 0, 0, -1, -1, 1, 0, 0, -1, 0, 1, 1, 0, -1, 0, 0, 1
\end{align*}
These $60$ numbers and their symmetric patterns are best understood with
\cref{fig:Phi5} which shows the first $70$ Hankel determinants. Values are plotted with  bullet points for the first antiperiod and with crosses afterwards. In particular one can visualize property \eqref{symetrie-De0}:
\begin{equation*}
\De_{j}=\De_{66-j}\quad \text{for }0≤j≤66,
\end{equation*}
which indicates a symmetry with respect to the point $(33,0)$ (labeled with a blank circle on the figure) inside the subset $\{(j,\De_j),\ 0≤j≤66\}$ of the graph.
\end{enumerate}
\end{ex}
\begin{figure}
  \includegraphics[width=\linewidth]{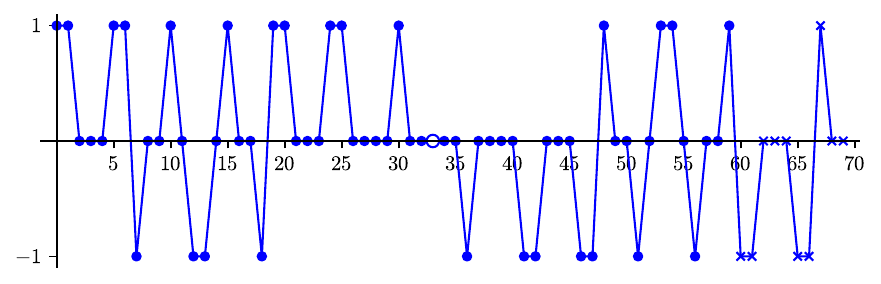}
  \caption{First $70$ Hankel determinants of $\Phi_5$.}
  \label{fig:Phi5}
\end{figure}

\begin{rem} \Cref{fig:Phi5} helps us to visualize the repartition of zero Hankel determinants, 
in particular we can observe the existence of subsets formed with  $k$ consecutive zeros for $k=1,2,3,4,5$. Using the results of \cref{value-sj}, it is not difficult to see that this is a general phenomenon: for any $k\in\{1,\ldots,n\}$ there exists a set of $k$ successive zero Hankel determinants, and  the case $k=n$ occurs exactly once, between the nonzero determinants $\De_{s_{3n-1}}$ and $\De_{s_{3n}}$. Actually this maximal set of $n$ successive zero determinants is also the only one which is invariant under the symmetry \eqref{symetrie-De0} (see \cref{fig:Phi5} for case $n=5$; the $5$ consecutive zero determinants are $\De_{31}$, $\De_{32}$, $\De_{33}$, $\De_{34}$ and $\De_{35}$).
\end{rem}

\begin{proof} 
First assertion in Part~(1) simply results from formula \eqref{HankelHF} and from the fact that $s_{6n-5}=2n(n+1)-1$ by Lemma~\ref{value-sj}. To establish all other statements we  split the  quite long proof into three main steps: (A) first we establish Part~(1) for $0≤q≤3n-1$, then (B) we justify the symmetry relation \eqref{symetrie-De0} of Part~(2), and finally (C) we will use this symmetry to prove Part~(1) in the remaining cases $3n≤q≤6n-5$. 
\smallbreak

\noindent\textbf{(A) Proof of Part~(1) for $0≤q≤3n-1$.} We argue on the congruences of $q$ modulo $3$ and take care separately of some special cases.
\smallbreak

$\bullet$ Suppose first that $q=3p+2$ with $0≤p≤n-2$ (hence $2≤q≤3n-4$). By \eqref{Hankel-prod} we must calculate
\begin{equation*}
\Delta_{s_{3p+2}}
=(-1)^{\eps_{3p+2}}(v_0 v_1\cdots v_{3p+1})^{s_{3p+2}}\,v_1^{s_1}v_2^{s_2}\cdots{}v_{3p+1}^{s_{3p+1}}
\end{equation*}
where, by \eqref{value-vj}, all $v_j=-1$ except $v_0=1$, so that $\Delta_{s_{3p+2}}=(-1)^{\mu_{3p+2}}$ with
\begin{equation*}
\mu_{3p+2}
:=\eps_{3p+2}+(3p+1)s_{3p+2}+\sum_{i=1}^{3p+1}s_i.
\end{equation*}
Using Lemmas~\ref{value-sj} and \ref{value-epsj} we  compute 
\begin{equation*}
\textstyle
\mu_{3p+2}
=\fr{(p+1)(n-p-1)(n-2)}{2}+\fr{p(p-1)(p+1)}{3}+(3p+1)(p+1)n+(p+1)\Bigl(1+\fr{p(3n+2)}{2}\Bigr)
\end{equation*}
and easily check its parity to establish the first case in \eqref{Des3p+2}.
\smallbreak

$\bullet$ Suppose that $q=3p+1$ with $0≤p≤n-2$ (hence $1≤q≤3n-5$). Let us remark that, by \eqref{Hankel-prod}, the following relation between two consecutive nonzero determinants $\De_{s_q}$ always holds:
\begin{equation}
\label{Deltasjsj+1}
\forall q≥0,\quad \De_{s_{q+1}}=\De_{s_q}(-1)^{\fr{k_q(k_q+1)}{2}}\bigl(v_0 v_1\cdots v_q\bigr)^{k_q+1}.
\end{equation}
In particular, 
\begin{align*}
\De_{s_{3p+2}}&=\De_{s_{3p+1}}(-1)^{\fr{k_{3p+1}(k_{3p+1}+1)}{2}}\bigl(v_0 v_1\cdots v_{3p+1}\bigr)^{k_{3p+1}+1}\\
&=(-1)^{(p+1)\fr{n(n-1)}{2}}(-1)^{\fr{(n-2-p)(n-1-p)}{2}}(-1)^{(3p+1)(n-1-p)}.
\end{align*}
by using \eqref{Des3p+2}, \eqref{value-kj} and \eqref{value-vj}. An easy calculation gives \eqref{Des3p+1} for $0≤p≤n-2$.
\smallbreak

$\bullet$ Suppose that $q=3p$ with $1≤p≤n-1$ (hence $3≤q≤3n-3$). Using \eqref{Deltasjsj+1},  \eqref{value-kj} and \eqref{value-vj}, we have
\begin{equation*}
\De_{s_{3p}}=\De_{s_{3(p-1)+2}}(-1)^{\fr{p(p-1)}{2}}(-1)^{p(3p-1)}.
\end{equation*}
Applying the first item above (since $0≤p-1≤n-2$) we obtain the first line in \eqref{Des3p}.
\smallbreak

$\bullet$ We are left with  three special cases: $q=0$, $q=3n-2$ and $q=3n-1$. For the first one, we see that the convention $\De_{s_0}=\De_0=1$ is coherent with formula \eqref{Des3p} when $p=0$. To obtain $\De_{s_{3n-2}}$ we apply \eqref{Deltasjsj+1} because the case $q=3n-3$ is already known, and then, the value of $\De_{s_{3n-1}}$ is obtained in the same way from the one of $\De_{s_{3n-2}}$.
\smallbreak

\noindent\textbf{(B) Proof of the symmetry formula \eqref{symetrie-De0}.} The starting point is to note that this formula is equivalent to the following one
\begin{equation}\label{symetrie-De0-bis}
\De_{s_q}=(-1)^{\fr{n(n+1)}{2}}\De_{s_{6n-1-q}}\quad \text{for }3n≤q≤6n-1.
\end{equation}
Indeed, $\De_j$ is nonzero if and only if $j=s_q$ for some $q≥0$, and
\begin{equation*}
(2n+1)(n+1)-s_q=s_{6n-1}-s_q=s_{6n-1-q}
\end{equation*}
by \eqref{sym2-s}. Therefore \eqref{symetrie-De0} is equivalent to \eqref{symetrie-De0-bis} for $0≤q≤6n-1$, and obviously it suffices to check the equality when $3n≤q≤6n-1$.

To do so, let us compare $\De_{s_q}$ and $\De_{s_{6n-1-q}}$ by using \eqref{Hankel-prod}. Actually we must treat the case $q=6n-1$ apart, since \eqref{Hankel-prod} does not make sense for $\De_{s_0}$. But we can check directly  relation \eqref{symetrie-De0-bis} by using Proposition~\ref{Desp+6n-4} and formula~\eqref{Des3p}:
\begin{equation*}
\De_{s_{6n-1}}=(-1)^n\De_{s_3}=(-1)^n(-1)^{\fr{n(n-1)}{2}}=(-1)^{\fr{n(n+1)}{2}}=(-1)^{\fr{n(n+1)}{2}}\De_{s_0}.
\end{equation*}
Now, let us continue our calculations in the generic case $3n≤q≤6n-2$ and use \eqref{Hankel-prod} for the Hankel determinants: 
\begin{equation*}
\fr{\De_{s_q}}{\De_{s_{6n-1-q}}}
=\fr{(-1)^{\eps_q}\bigl(v_0 v_1\cdots v_{q-1}\bigr)^{s_q}\,\bigl(v_1^{s_1}v_2^{s_2}\cdots v_{q-1}^{s_{q-1}}\bigr)}
{(-1)^{\eps_{6n-1-q}}\bigl(v_0 v_1\cdots v_{6n-2-q}\bigr)^{s_{6n-1-q}}\,\bigl(v_1^{s_1}v_2^{s_2}\cdots v_{6n-2-q}^{s_{6n-2-q}}\bigr)}.
\end{equation*}
Under our assumption we have $6n-1-q≤3n-1≤q-1$. Thus some products are canceling and we further simplify this expression by using successively \eqref{sym2-eps}, the fact that $v_j\in\{\pm 1\}$ and \eqref{sym2-s}:
\begin{align*}
\fr{\De_{s_q}}{\De_{s_{6n-1-q}}}
&=(-1)^{\eps_{6n-1}}\bigl(v_0 v_1\cdots v_{6n-2-q}\bigr)^{s_q-s_{6n-1-q}}
\bigl(v_{6n-1-q}\cdots v_{q-1}\bigr)^{s_q}
\bigl(v_{6n-1-q}^{s_{6n-1-q}}\cdots v_{q-1}^{s_{q-1}}\bigr)
\\
&=(-1)^{\eps_{6n-1}}\bigl(v_0 v_1\cdots v_{6n-2-q}\bigr)^{s_q+s_{6n-1-q}}
\bigl(v_{6n-1-q}\cdots v_{q-1}\bigr)^{s_q}
\bigl(v_{6n-1-q}^{s_{6n-1-q}}\cdots v_{q-1}^{s_{q-1}}\bigr)
\\
&=(-1)^{\eps_{6n-1}}\bigl(v_0 v_1\cdots v_{6n-2-q}\bigr)^{s_{6n-1}}
\bigl(v_{6n-1-q}\cdots v_{q-1}\bigr)^{s_q}
\bigl(v_{6n-1-q}^{s_{6n-1-q}}\cdots v_{q-1}^{s_{q-1}}\bigr).
\end{align*}
Recall from \eqref{value-vj} that all $v_j$'s involved in this formula are equal to $-1$ except  $v_0=v_{3n-2}=v_{3n+1}=1$. Hence
\begin{equation*}
\bigl(v_0 v_1\cdots v_{6n-2-q}\bigr)^{s_{6n-1}}=(-1)^{(6n-2-q)s_{6n-1}}=(-1)^{q s_{6n-1}}
\end{equation*}
because $6n-2-q<3n-2$.
Moreover, by \eqref{sym2-eps} we have
\begin{equation*}
(-1)^{\eps_{6n-1}}=(-1)^{\fr{n(n+1)(2n+1)}{6}}=(-1)^{\sum_{k=1}^n k^2}=(-1)^{\sum_{k=1}^n k}=(-1)^{\fr{n(n+1)}{2}}.
\end{equation*}
Thus we have proved so far that
\begin{equation}\label{xoen}
\fr{\De_{s_q}}{\De_{s_{6n-1-q}}}
=(-1)^{\fr{n(n+1)}{2}}(-1)^{q s_{6n-1}} 
\bigl(v_{6n-1-q}\cdots v_{q-1}\bigr)^{s_q}
\bigl(v_{6n-1-q}^{s_{6n-1-q}}\cdots v_{q-1}^{s_{q-1}}\bigr).
\end{equation}
To proceed we must distinguish between several cases.
\smallbreak

$\bullet$ Suppose first that $q≥3n+2$. On the one hand,
\begin{equation*}
\bigl(v_{6n-1-q}\cdots v_{q-1}\bigr)^{s_q}
=(-1)^{(2q-6n-1)s_q}=(-1)^{s_q}
\end{equation*}
because  in this product all terms but two equal $-1$.
On the other hand, 
\begin{equation*}
v_{6n-1-q}^{s_{6n-1-q}}\cdots v_{q-1}^{s_{q-1}}=(-1)^{\al_q}
\end{equation*}
with
\begin{align*}
\al_q&:=s_{6n-1-q}+\cdots +s_{q-1}-s_{3n-2}-s_{3n+1}\\
&=\bigl(s_{6n-1-q}+\cdots +s_{3n-3}+s_{3n-1}+s_{3n}+s_{3n+2}+\cdots +s_{q-1}+s_q\bigr)-s_q.
\end{align*}
The reason to write last expression is that, inside the parenthesis, all terms are pairing to form sums of the type $s_k+s_{6n-1-k}$ which all equal $s_{6n-1}$ by \eqref{sym2-s}. There are $q-3n$ such pairings, so that
\begin{equation*}
v_{6n-1-q}^{s_{6n-1-q}}\cdots v_{q-1}^{s_{q-1}}=(-1)^{(q-3n)s_{6n-1}}(-1)^{s_q}.
\end{equation*}
Finally we obtain that
\begin{equation*}
\fr{\De_{s_q}}{\De_{s_{6n-1-q}}}
=(-1)^{\fr{n(n+1)}{2}}(-1)^{q s_{6n-1}}(-1)^{s_q}
(-1)^{(q-3n)s_{6n-1}}(-1)^{s_q}
=(-1)^{\fr{n(n+1)}{2}}(-1)^{n s_{6n-1}}
\end{equation*}
with $s_{6n-1}=(n+1)(2n+1)$, hence \eqref{symetrie-De0-bis}.
\smallbreak

$\bullet$ Suppose that $q=3n+1$. Then
\begin{equation*}
\bigl(v_{6n-1-q}\cdots v_{q-1}\bigr)^{s_q}
=\bigl(v_{3n-2}v_{3n-1} v_{3n}\bigr)^{s_{3n+1}}
=1
\end{equation*}
and
\begin{equation*}
v_{6n-1-q}^{s_{6n-1-q}}\cdots v_{q-1}^{s_{q-1}}
=v_{3n-2}^{s_{3n-2}}v_{3n-1}^{s_{3n-1}} v_{3n}^{s_{3n}}
=(-1)^{s_{3n-1}+s_{3n}}=(-1)^{s_{6n-1}}.
\end{equation*}
Thus \eqref{xoen} yields
\begin{equation*}
\fr{\De_{s_{3n+1}}}{\De_{s_{3n-2}}}
=(-1)^{\fr{n(n+1)}{2}}(-1)^{(3n+1) s_{6n-1}}(-1)^{s_{6n-1}}
=(-1)^{\fr{n(n+1)}{2}}(-1)^{n s_{6n-1}}
\end{equation*}
with $s_{6n-1}=(n+1)(2n+1)$, hence \eqref{symetrie-De0-bis} when $q=3n+1$.
\smallbreak

$\bullet$ When $q=3n$, the proof is similar to the preceding case.
\smallbreak

\noindent\textbf{(C) Proof of Part~(1) for $3n≤q≤6n-5$.} We now apply  symmetry \eqref{symetrie-De0} to the formulas of Part~(1) already proved for $0≤q≤3n-1$ to deduce the remaining formulas, when $3n≤q≤6n-5$.
\smallbreak

$\bullet$ Suppose that $q=3p$ with $n+1≤p≤2n-2$ (hence $3n+3≤q≤6n-6$). By \eqref{symetrie-De0},
\begin{equation*}
\De_{s_{3p}}=(-1)^{\fr{n(n+1)}{2}}\De_{s_{3(2n-p-1)+2}}
\end{equation*}
with $0≤2n-p-1≤n-2$. Hence, by \eqref{Des3p+2} for Case~(A) we have
\begin{equation*}
\De_{s_{3p}}
=(-1)^{\fr{n(n+1)}{2}}(-1)^{(2n-p)\fr{n(n-1)}{2}}
=(-1)^{\fr{n(n+1)}{2}}(-1)^{p\fr{n(n-1)}{2}}
=(-1)^{(p+1)\fr{n(n-1)}{2}}(-1)^n,
\end{equation*}
as required.
\smallbreak

$\bullet$ Suppose that $q=3p+1$ with $n≤p≤2n-2$ (hence $3n+1≤q≤6n-5$). By \eqref{symetrie-De0},
\begin{equation*}
\De_{s_{3p+1}}=(-1)^{\fr{n(n+1)}{2}}\De_{s_{3(2n-p-1)+1}}
\end{equation*}
with $1≤2n-p-1≤n-2$. Hence, by \eqref{Des3p+1} for Case~(A),
\begin{align*}
\De_{s_{3p+1}}
&=(-1)^{\fr{n(n+1)}{2}}(-1)^{(2n-p-1)\fr{n(n-1)}{2}}(-1)^{\fr{(2n-p-1)(2n-p)}2}\\
&=(-1)^{\fr{n(n+1)}{2}}(-1)^{(p+1)\fr{n(n-1)}{2}}(-1)^{-n+\fr{p(p+1)}2}\\
&=(-1)^{p\fr{n(n-1)}{2}}(-1)^{\fr{p(p+1)}2}
\end{align*}
which is again \eqref{Des3p+1}.
\smallbreak

$\bullet$ Suppose that $q=3p+2$ with $n≤p≤2n-3$ (hence $3n+2≤q≤6n-7$). This case is similar to the previous one and we skip details.
\smallbreak

$\bullet$ The only remaining case is $q=3n$, which follows immediately from \eqref{symetrie-De0} and from the value of $\De_{s_{3n-1}}$ given by the second row in \eqref{Des3p+2}.
\end{proof}

\section{The Gale-Robinson relations}\label{sec-GR}

Assume as usual that $n≥3$. For any integer $j≥0$ and any $\ell\in\{0,1,\ldots,n+1\}$, set
\begin{equation*}
\Ga_j^{(\ell)}:=\De_j^{(\ell)}\, \De_{j+2n+2}^{(\ell)} - \De_{j+1}^{(\ell)}\,\De_{j+2n+1}^{(\ell)} + \bigl(\De_{j+n+1}^{(\ell)}\bigr)^2.
\end{equation*}
We want to prove the Gale-Robinson recurrence stated in \cref{MainGR}, which reads
\begin{equation}\label{GRbis}
\Ga_j^{(\ell)}=0\qquad \text{for all $j≥0$ and all $\ell\in\{0,1,\ldots,n+1\}$}.
\end{equation}
Assuming the validity of \cref{MainContiguity} (to be proved in \cref{section-shifts}) it is easily checked that
\begin{equation*}
\Ga_j^{(\ell+1)}=\Ga_j^{(\ell)},
\end{equation*}
so that it suffices to prove \eqref{GRbis} in the case $\ell=0$, and this will be our purpose in this section. 

For simplicity we put $\Ga_j:=\Ga_j^{(0)}$. 
Because of the (anti-)periodicity property  shown in \cref{MainValPer}, it remains to show that 
\begin{equation}\label{GR2}
\Ga_j=0\quad\text{for }0≤j≤2n(n+1)-1,
\end{equation}
where
\begin{equation*}
\Ga_j:=\De_j\, \De_{j+2n+2} - \De_{j+1}\,\De_{j+2n+1} + \bigl(\De_{j+n+1}\bigr)^2.
\end{equation*}
By \cref{value-sj} we know which Hankel determinants are nonzero and, by the previous theorem, we know the explicit values of these nonzero determinants, so we just have to check that $\Ga_j$  vanishes in any case. Adding and multiplying five Hankel determinants $\De_i$ which can be either $-1$, $0$ or $1$ is certainly a triviality but the difficulty comes from the fact that we have to handle them simultaneously, and there is no clue of how their values are connected in general. In fact, we will distinguish four cases, looking at pairs of successive determinants $(\De_j,\De_{j+1})$ and discussing on the  nullity of one of them, in order to simplify the expression of $\Ga_j$ when possible. And, because \eqref{GR2} must be checked for $j≤2n(n+1)-1$ we will have to discuss on Hankel determinants $\De_i$ with  index 
$$i=j+2n+2≤2n(n+1)-1+2n+2=2n(n+2)+1,$$
i.e. we need the description of the intersection $\cS\cap\{0,1,\ldots,2n(n+2)+1\}$. One the one hand, Lemma~\ref{value-sj} implies the characterization:
\begin{equation}\label{caracS}
j\in\cS\text{ and }0≤j≤2n(n+1)\iff j\text{ satisfies one of the conditions (i)-(v)}
\end{equation}
where
\begin{enumerate}[label=(\roman*)]
\item $j\equiv 0\mod {n+1}$ and $0≤j≤2n(n+1)$
\item $j\equiv 1\mod {n+1}$ and $ 1≤j≤n^2$
\item $j\equiv n\mod {n+1}$ and $n+(n+1)^2≤j≤n+(2n-1)(n+1)$
\item $j\equiv 0\mod {n}$ and $n≤j≤(n-1)n$
\item $j\equiv 1\mod {n}$ and $1+(n+4)n≤j≤1+(2n+1)n$.
\end{enumerate}
On the other hand, the $j\in\cS$ such that $2n(n+1)≤j≤2n(n+2)+1$  can be found by Lemma~\ref{sp+6n-4}: they are such that $j-2n(n+1)\in\cS$ and characterization \eqref{caracS} can be applied to these translated numbers.

\begin{ex} When $n=5$, according to \eqref{caracS} we must determine the $j\in\cS$ such that $0≤j≤60$. By the results of \cref{value-sj} there are $27$ of them, equally divided into three subsets:
\begin{align*}
\cS_0&:=\{s_{3p},\ 0≤p≤8\}=\{0,6,12,18,24,36,46,51,56\},\\
\cS_1&:=\{s_{3p+1},\ 0≤p≤8\}=\{1,7,13,19,25,41,47,53,59\},\\
\cS_2&:=\{s_{3p+2},\ 0≤p≤8\}=\{5,10,15,20,30,42,48,54,60\},
\end{align*}
and symmetries read  as follows:
\begin{equation*}
s_{j+16}=s_j+41\quad \text{for }0≤j≤13,\qquad s_j+s_{29-j}=66\quad\text{for }0≤j≤29.
\end{equation*}
One can visualize the set $\cS\cap\{0,1,\ldots,60\}=\cS_0\cup\cS_1\cup\cS_2$ and all the symmetries on \cref{fig:S}. 
\end{ex}
\begin{figure}
  \includegraphics[width=\linewidth]{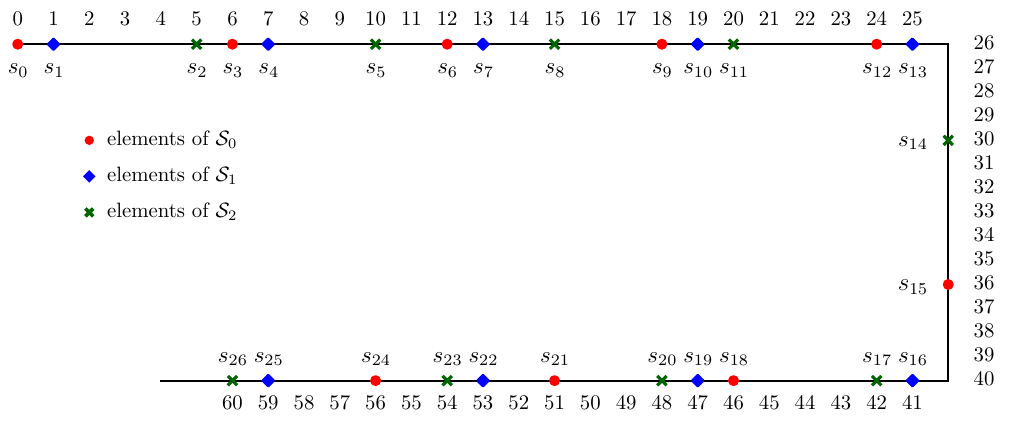}
  \caption{The set $\cS\cap\{0,1,\ldots,2n(n+1)\}=\cS_0\cup\cS_1\cup\cS_2$ when $n=5$.}
  \label{fig:S}
\end{figure}

Now, let us begin the proof of the Gale-Robinson recurrences \eqref{GR2}. Computations are tedious, but easy and similar in the first three cases we will look at. Therefore, full details will be provided in the first and fourth cases only.
\smallbreak

\noindent\textbf{Case 1: $\De_j\neq 0$ and $\De_{j+1}\neq 0$.} According to the previous discussion  this situation corresponds to the following five possibilities:
\begin{enumerate}[label=(\alph*)]
\item $j=n=s_2$,
\item $j=p(n+1)=s_{3p}$ for $0≤p≤n-1$,
\item $j=n^2-n-1=s_{3n-5}$,
\item $j=n^2+4n+1=s_{3n+3}$,
\item $j=n+(p+1)(n+1)=s_{3p+1}$ for $n≤p≤2n-2$,
\end{enumerate}
and we go for a case-by-case examination.

\begin{enumerate}[label=(\alph*), wide]
\item If $j=n=s_2$ then
 $j+1=n+1=s_3$, $j+n+1=2(n+1)=s_6$, $j+2n+2=3(n+1)=s_9$ and $j+2n+1=3n+2\notin\cS$, so that, \begin{equation*}
\Ga_j=\De_{s_2}\De_{s_9}+\De_{s_6}^2=(-1)^{\fr{n(n-1)}{2}}(-1)^{\fr{3n(n-1)}{2}}(-1)^3+1=0.
\end{equation*}
\item Suppose that $j=p(n+1)=s_{3p}$ for $0≤p≤n-1$.  
\begin{itemize}[wide]
\item We look first at the generic situation $0≤p≤n-3$. Then $j+1=p(n+1)+1=s_{3p+1}$, $j+n+1=(p+1)(n+1)=s_{3(p+1)}$, $j+2n+2=(p+2)(n+1)=s_{3(p+2)}$ and $j+2n+1=n+(p+1)(n+1)\notin\cS$. Thus
\begin{equation*}
\Ga_j=\De_{s_{3p}}\De_{s_{3(p+2)}}+\De_{s_{3(p+1)}}^2
=(-1)^{(p+2)\fr{n(n-1)}{2}}(-1)^{\fr{(p+2)(p+1)}2}
(-1)^{p\fr{n(n-1)}{2}}(-1)^{\fr{p(p-1)}2}
+1=0.
\end{equation*}
\item If $j=(n-2)(n+1)=s_{3(n-2)}$, then we have $j+1=s_{3(n-2)+1}$, $j+n+1=s_{3(n-1)}$, $j+2n+2=s_{3(n-1)+2}$ and $j+2n+1=n+(p+1)(n+1)\notin\cS$. Thus
\begin{equation*}
\Ga_j=\De_{s_{3(n-2)}}\De_{s_{3(n-1)+2)}}+\De_{s_{3(n-1)}}^2
=(-1)^{(n-2)\fr{n(n-1)}{2}}(-1)^{\fr{(n-2)(n-3)}2}
(-1)^{\fr{n(n-1)^2}2}
+1=0.
\end{equation*}
\item If $j=(n-1)(n+1)=s_{3n-3}$, then $j+1=s_{3n-2}$, $j+n+1=s_{3n-1}$, $j+2n+2=s_{3n}$ and $j+2n+1=n^2+2n\notin\cS$. Hence
\begin{equation*}
\Ga_j=\De_{s_{3n-3}}\De_{s_{3n}}+\De_{s_{3n-1}}^2
=(-1)^{\fr{n(n-1)^2}{2}}(-1)^{\fr{(n-1)(n-2)}2}
(-1)^{\fr{n(n+1)(n+2)}2}
+1=0.
\end{equation*}
\end{itemize} 
\item When $j=n^2-n-1=s_{3n-5}$ we have $j+1=s_{3n-4}$, $j+n+1=s_{3n-2}$, $j+2n+1=s_{3n-1}$ and $j+2n+2=n^2+n+1\notin\cS$, so that
\begin{equation*}
\Ga_j=-\De_{s_{3n-4}}\De_{s_{3n-1}}-\De_{s_{3n-1}}^2
=(-1)^{\fr{n(n-1)^2}{2}}(-1)^{\fr{n(n-1)^2}{2}}-1
=0.
\end{equation*}
\item If $j=n^2+4n+1=s_{3n+3}$, then $j+1=s_{3n+4}$ but $j+n+1$, $j+2n+1$ and $j+2n+2$ are not in $\cS$ and give zero determinants. Trivially, $\Ga_j=0$ as well.
\item Assume that $j=n+(p+1)(n+1)=s_{3p+1}$ for $n≤p≤2n-2$.
\begin{itemize}[wide]
\item We suppose first that  $n≤p≤2n-4$. In this case, $j+2n+2≤n+(2n-1)(n+1)<2n(n+1)$ so that all indices remain inside the first (anti-)period: we compute as before $j+1=s_{3p+2}$, $j+n+1=s_{3(p+1)+1}$, $j+2n+2=s_{3(p+2)+1}$ and observe that $j+2n+1=2n+(p+2)(n+1)\notin \cS$ (look at congruences modulo $n$ and $n+1$). Thus
\begin{align*}
\Ga_j&=\De_{s_{3p+1}}\De_{s_{3(p+2)+2)}}+\De_{s_{3(p+1)+1}}^2\\
&=(-1)^{p\fr{n(n-1)}{2}}(-1)^{\fr{p(p+1)}2}
(-1)^{(p+2)\fr{n(n-1)}{2}}(-1)^{\fr{(p+2)(p+3)}2}+1\\
&=0.
\end{align*}
\item Suppose that $p=2n-3$. Then $j+1=s_{3(2n-3)+2}$, $j+n+1=s_{3(2n-2)+1}$,   $j+2n+1=n+2n(n+1)$ and $j+2n+2=n+2n(n+1)$. This time, the last two indices exceed the (anti-)period length $2n(n+1)$, so we use Lemma~\ref{sp+6n-4} to see that $j+2n+1\notin \cS$ and $j+2n+2=s_2+2n(n+1)$. For the latter,  the corresponding determinant is  calculated with property \eqref{per-De0}:
\begin{equation*}
\De_{n+2n(n+1)}=(-1)^n\De_n=(-1)^n\De_{s_2}.
\end{equation*}
This yields 
\begin{equation*}
\Ga_j=\De_{s_{3(2n-3)+2}}(-1)^n\De_{s_2}+1
=(-1)^{(2n-3)\fr{n(n-1)}{2}}(-1)^{\fr{(2n-3)(2n-2)}2}(-1)^n(-1)^{\fr{n(n-1)}{2}}+1=0.
\end{equation*}
\item Suppose now that $p=2n-2$. Then $j+1=s_{3(2n-2)+2}$ and all other indices exceed the (anti-)period length: $j+n+1=n+2n(n+1)$,  $j+2n+1=2n+2n(n+1)$ and $j+2n+2=2n+1+2n(n+1)$. Using Lemma~\ref{sp+6n-4} and \eqref{per-De0} we get:
\begin{align*}
\De_{n+2n(n+1)}&=(-1)^n\De_n=(-1)^n\De_{s_2},\\
\De_{2n+2n(n+1)}&=(-1)^n\De_{2n}=(-1)^n\De_{s_5},\\
\De_{2n+1+2n(n+1)}&=(-1)^n\De_{2n+1}=0.
\end{align*}
Finally, we find again that
\begin{equation*}
\Ga_j=-\De_{s_{3(2n-2)+2}}(-1)^n\De_{s_5}+1
=-(-1)^n (-1)^n (-1)^{n(n-1)}+1=0.
\end{equation*}
and this completes the proof in Case~1. 
\end{itemize}
\end{enumerate}

\noindent\textbf{Case 2: $\De_j\neq 0$ and $\De_{j+1}= 0$.} Since  verifications are similar to Case~1 we will only provide the possible values of the indices $j,j+1,j+n+1,j+2n+1,j+2n+2$; then, using Theorem~\ref{thm-Hankel0} it is  easy to check that one always has $\Ga_j=0$. 

Let us describe the situation. According to \eqref{caracS} one has seven possibilities for $j$ which are listed below. In each case we indicate the corresponding values only for indices $j+n+1$ and $j+2n+2$: the ones for $j+1$ and $j+2n+1$ are useless since by hypothesis $\De_{j+1}\De_{j+2n+1}=0$.
\begin{enumerate}[label=(\alph*)]
\item $j=1+p(n+1)=s_{3p+1}$  for $0≤p≤n-3$. Then $j+n+1=s_{3(p+1)+1}$ and $j+2n+2=s_{3(p+2)+1}$.
\item $j=(p+1)n=s_{3p+2}$ for $1≤p≤n-2$. Then $j+n+1=(p+2)n+1\notin\cS$ and $j+2n+2=(p+3)n+2\notin\cS$.
\item $j=n^2=s_{3(n-1)+1}$. Then $j+n+1=n^2+n+1\notin\cS$ and $j+2n+2=n^2+2n+2\notin\cS$.
\item $j=n(n+1)=s_{3(n-1)+2}$. Then $j+n+1=s_{3n}$ and $j+2n+2s_{3n+2}$.
\item $j=(n+1)^2=s_{3n}$. Then $j+n+1=s_{3n+2}$ and $j+2n+2=s_{3(n+1)+2}$.
\item $j=(p+2)(n+1)=s_{3p+2}$ for $n≤p≤2n-3$. Then $j+n+1=s_{3(p+1)+2}$ and $j+2n+2=s_{3(p+2)+2}$. 
\item $j=1+(p+3)n=s_{3p}$ for $n+2≤p≤2n-2$. Then $j+n+1=2+(p+4)n\notin\cS$ and $j+2n+2=3+(p+5)n\notin\cS$.
\end{enumerate}
N.B. For subcases~(f) and (g), one must use Lemma~\ref{sp+6n-4} and property \eqref{per-De0} when indices are $≥2n(n+1)$. Subcase~(g) can only occur when $n≥4$.
\smallbreak

\noindent\textbf{Case 3: $\De_j= 0$ and $\De_{j+1}\neq 0$.} Again, the situation is very similar to the first ones. With \eqref{caracS} we see that this case corresponds to  seven possibilities for $j+1$. Since $\De_j\De_{j+2n+2}$ always vanishes by hypothesis, we give only the corresponding values of $j+n+1$ and $j+2n+1$.
\begin{enumerate}[label=(\alph*)]
\item $j+1=(p+1)n=s_{3p+2}$ for $0≤p≤n-3$. Then $j+n+1=s_{3(p+1)+2}$, and $j+2n+1=s_{3(p+2)+2}$ except if $p=n-3$ in which case $j+2n+1=n^2=s_{3(n-1)+1}$.
\item $j+1=p(n+1)=s_{3p}$  for $2≤p≤n-1$. Then $j+n+1=(p+1)n+p\notin\cS$ and $j+2n+1=(p+2)n+p\notin\cS$.
\item $j+1=n(n+1)=s_{3(n-1)+2}$. Then $j+n+1=n(n+2)\notin\cS$ and $j+2n+1=n(n+3)\notin\cS$.
\item $j+1=(n+1)^2=s_{3n}$. Then $j+n+1=(n+1)^2+n=s_{3n+1}$ and $j+2n+1=1+n(n+4)=s_{3(n+1)}$.
\item $j+1=(n+1)^2+n=s_{3n+1}$. Then $j+n+1=1+n(n+4)=s_{3(n+1)}$ and $j+2n+1=1+n(n+5)=s_{3(n+2)}$.
\item $j+1=1+(p+3)n=s_{3p}$ for $n+1≤p≤2n-2$. Then $j+n+1=1+(p+4)n=s_{3(p+1)}$ and $j+2n+1=1+(p+5)n=s_{3(p+2)}$.
\item $j+1=n+(p+1)(n+1)=s_{3p+1}$ for $n+2≤p≤2n-2$. Then $j+n+1=2n+(p+1)(n+1)\notin\cS$ and $j+2n+1=3n+(p+1)(n+1)\notin\cS$.
\end{enumerate}
N.B. For subcases~(f) and (g), one must use Lemma~\ref{sp+6n-4} and property \eqref{per-De0} when indices are $≥2n(n+1)$. Subcase~(g) can only occur when $n≥4$.
\smallbreak

\noindent\textbf{Case 4: $\De_j= 0=\De_{j+1}$.} In this case, the Gale-Robinson recurrence \eqref{GR2} will be true if and only if $\De_{j+n+1}$ always vanishes. Assume on the contrary that $\De_{j+n+1}\neq 0$, i.e. $j+n+1\in\cS$. We shall use again \eqref{caracS} and distinguish four subcases.
\begin{enumerate}[label=(\alph*),wide]
\item Suppose that $j+n+1≤s_{3n}=(n+1)^2$, i.e. $j≤n(n+1)=s_{3n-1}$. According to \eqref{caracS}, we must have
\begin{equation*}
j+n+1\equiv 0\mod{n+1}\quad\text{or}\quad j+n+1\equiv 1\mod{n+1}
\quad\text{or}\quad j+n+1\equiv 0\mod{ n}
\end{equation*}
which amounts to say that $j$ fullfills (i)  or (ii) in \eqref{caracS} (first two cases) or that $j+1$ satisfies (iv) in \eqref{caracS} (last case). This implies in return that $j\in\cS$ or $j+1\in\cS$,  i.e. $\De_j\neq 0$ or $\De_{j+1}\neq 0$, which is a contradiction.
\item Suppose that $j+n+1$ is one of the following three numbers:
\begin{equation*}
s_{3n+1}=n+(n+1)^2,\quad s_{3n+2}=(n+2)(n+1),\quad s_{3n+3}=1+(n+4)n.
\end{equation*} 
In the first case, we have $j+1=(n+1)^2\in\cS$ by \eqref{caracS}, hence $\De_{j+1}\neq 0$, contradiction. In the second case, we have $j=(n+1)^2\in\cS$ and $\De_j\neq 0$, also a contradiction. In the third case we have $j+1=n+(n+1)^2\in\cS$  and $\De_{j+1}\neq 0$, still a contradiction.
\item Suppose now that $n+(n+1)(n+2)=s_{3n+4}≤j+n+1≤s_{6n-4}=2n(n+1)$. Then $j+n+1$ satisfies conditions (i) or (iii) or (v) of \eqref{caracS}, which implies that either $j$ satisfies (i) or $j+1$ satisfies (i) or $j+1$ satisfies (v). This leads in all cases to $\De_j=0$ or $\De_{j+1}=0$, a contradiction.
\item To finish the proof, suppose that $j+n+1≥s_{6n-3}=2n^2+2n+1$ (see \eqref{s6n-3}). Since $j≤2n(n+1)-1$ one also has $j+n+1≤(2n+1)(n+1)-1=2n^2+3n=s_{6n-2}$ by \eqref{s6n-2}. In other words, $j+n+1\in\{s_{6n-2},s_{6n-3}\}=\{2n^2+2n+1,2n^2+3n\}$. Therefore, either $j=2n^2+n+1=s_{6n-6}$ or $j=2n^2+2n-1=s_{6n-5}$. In both cases $\De_j\neq 0$, contradiction.
\end{enumerate}
\smallbreak
We thus have completely established \cref{MainGR} in case $\ell=0$.

\section{The shifted case}
\label{section-shifts}

In this section we give a proof of \cref{MainContiguity,MainModulop}. Let us recall that the contiguity relations \eqref{contiguity} contained in \cref{MainContiguity} allow us to reduce the proofs of \cref{MainValPer,MainGR} to the case $\ell=0$ that has been treated in the previous sections. 

\subsection{The $H$-fraction expansion for $\Phi_n^{(\ell)}$ when $\ell\in\{0,1,\ldots,n+1\}$}

The contiguity relations \eqref{contiguity} stress the links which exist between  Hankel determinants  
\begin{equation*}
\De_j^{(\ell-1)}=\De_j^{(\ell-1)}(\Phi_n)=\De_j\bigl(\Phi_n^{(\ell-1)}\bigl)
\quad\text{and}\quad
\De_j^{(\ell)}=\De_j^{(\ell)}(\Phi_n)=\De_j\bigl(\Phi_n^{(\ell)}\bigl)
\end{equation*} 
of two consecutive shifts $\Phi_n^{(\ell-1)}$ and $\Phi_n^{(\ell)}$ of $\Phi_n$. To establish these relations we will first state the analogue of \cref{MainHF} for the shifted function $\Phi_n^{(\ell)}$.

\begin{lem}\label{lemma:cut}
Suppose that
$$F(q) = f_0 + f_1 q + f_2 q^2 + \cdots$$
is a quadratic power series satisfying the equation
\begin{equation}\label{EFQ}
A + B F + C F^2 =0
\end{equation}
where $A,B,C$ are polynomials (with $C≠0$). Then the shifted series
$$F^{(1)}(q) = f_1  + f_2 q + \cdots = \fr{F(q) -f_0}{q} $$
is also quadratic, satisfying the equation
$$
	A' + B' F^{(1)} + C' \bigl(F^{(1)}\bigr)^2 =0
$$
with $A',B',C'$  polynomials given by
\begin{align*}
A' &=\fr 1{q}(A+f_0B + f_0^2 C),\\
B'&=B + 2f_0C,\\
C'&=q C.
\end{align*}
\end{lem}

\begin{proof} 
Inserting $F=qF^{(1)}+f_0$ in equation \eqref{EFQ} gives immediately 
\begin{equation}\label{xorbu}
A+f_0B + f_0^2 C+q(B + 2f_0C)F^{(1)}+q^2 C \bigl(F^{(1)}\bigr)^2=0.
\end{equation}
Now, the constant term in the left-hand side of \eqref{EFQ} is $A(0)+B(0)F(0)+C(0)F(0)^2$, and this is actually also the constant term of the polynomial $A+f_0B + f_0^2 C$. Thus it equals zero, and the whole left-hand side of \eqref{xorbu} can be divided by $q$.
\end{proof}

We use this lemma to derive the following generalization of \eqref{FE-Phin} which gives the quadratic equations satisfied by the shifts of $\Phi_n$.

\begin{prop}\label{lemma:ABCshift}
Let $\ell\in\{0,1,\ldots,n+1\}$. The $\ell$-shift $\Phi_n^{(\ell)}$ of the $q$-metallic number $\Phi_n$ is also a quadratic power series with integer coefficients, satisfying the equation
\begin{equation}\label{FE-Phin-shift}
A^{(\ell)} +B^{(\ell)} \Phi_n^{(\ell)} + C^{(\ell)} \bigl(\Phi_n^{(\ell)}\bigr)^2 =0,
\end{equation}
where $A^{(\ell)},B^{(\ell)},C^{(\ell)}$ are polynomials with integer coefficients given as follows: for $0\leq \ell \leq n$,
\begin{align*}
A^{(\ell)} &=\frac{  q^{\ell+1} - (q^2 - q + 1)( q^{n}- q^{n-\ell} +1)    }{ (q-1)^2},\\
B^{(\ell)} &=\frac{2q^{\ell+1} - ( q^{2} - q+1)(q^n+1)  }{q - 1},\\
C^{(\ell)} &=q^{\ell+1},
\end{align*}
and for $\ell=n+1$,
\begin{align*}
A^{(n+1)} &=-q^{n-1},\\
B^{(n+1)} &=-\frac{( q^{2} - q+1) + ( q^{2} - 3q+1)q^{n}   }{q - 1},\\
C^{(n+1)} &=q^{n+2}.
\end{align*}
\end{prop}

\begin{proof} We proceed by induction on $\ell\in\{0,\ldots,n+1\}$. 

$\bullet$ We already know that \eqref{FE-Phin-shift} is true for $\ell=0$, by \eqref{FE-Phin}. 

$\bullet$ Assume that \eqref{FE-Phin-shift} is valid for some $\ell\in\{0,\ldots,n-1\}$. By \eqref{SE-Phin}, the constant term of $\Phi_n^{(\ell)}$ equals $1$. According to \cref{lemma:cut} we have
\begin{equation*}
A^{(\ell+1)} +B^{(\ell+1)} \Phi_n^{(\ell)} + C^{(\ell+1)} \bigl(\Phi_n^{(\ell)}\bigr)^2 =0
\end{equation*}
with
\begin{equation*}
C^{(\ell+1)} =\fr 1{q}\bigl(A^{(\ell)}+B^{(\ell)} +  C^{(\ell)}\bigr),\quad
B^{(\ell+1)}=B^{(\ell)} + 2C^{(\ell)},\quad
C^{(\ell+1)}=q C^{(\ell)}.
\end{equation*}
A simple calculation ensures that these formulas for $A^{(\ell+1)},B^{(\ell+1)},C^{(\ell+1)}$  coincide actually with the expressions given in the statement, at rank $\ell+1$. 

$\bullet$ Finally, when $\ell=n$ we have
\begin{equation*}
A^{(n)} =- q^{n},\quad
B^{(n)} =-\frac{ ( q^{2} - q+1) + ( q^{2} - 3q+1)q^n }{q - 1},\quad
C^{(n)} =q^{n+1}.
\end{equation*}
We apply again \cref{lemma:cut}, but this time with $f_0=0$ since the constant term of $\Phi_n^{(n)}$ equals $0$. We easily see that $A^{(n+1)},B^{(n+1)},C^{(n+1)}$ admit the required expressions.
\end{proof}

The key point is that we have already encountered the family of polynomials $A^{(\ell)},B^{(\ell)},C^{(\ell)}$. Indeed, with notations of \cref{defi-A_1,defi-A_2} and \eqref{defQj}:
\begin{align*}
A^{(0)}&= A_1^{(1)}=A_0,&
B^{(0)}&= B_1^{(1)}=B_0,&
C^{(0)}&= C_1^{(1)}=C_0,\\
A^{(\ell)}&= -A_1^{(3\ell+1)}=-A_{3\ell}, &
B^{(\ell)}&= B_1^{(3\ell+1)}=B_{3\ell}, &
C^{(\ell)}&= -C_1^{(3\ell+1)}=-C_{3\ell},\\
&\qquad\text{for }\ell = 1, 2, \ldots, n-2,\\
A^{(n-1)}&= -A_2^{(1)}=-A_{3n-3},&
B^{(n-1)}&= B_2^{(1)}=B_{3n-3}, &
C^{(n-1)}&= -C_2^{(1)}=-C_{3n-3},\\
A^{(n)}&= -A_2^{(3)}=-A_{3n-1}, &
B^{(n)}&= B_2^{(3)}=B_{3n-1}, &
C^{(n)}&= -C_2^{(3)}=-C_{3n-1},\\
A^{(n+1)}&= -A_2^{(4)}=-A_{3n}, &
B^{(n+1)}&= B_2^{(4)}=B_{3n}, &
C^{(n+1)}&= -C_2^{(4)}=-C_{3n}.
\end{align*}
(Notice that the minus signs appearing here are due to the minus sign in the first numerator of a $H$-fraction constructed with the  algorithm \alg, see \eqref{Hfracalgo}.)
From these identities we deduce that the $H$-fraction expansions of the shifts $\Phi_n^{(\ell)}$,  $1≤\ell≤n+1$, can be borrowed from \cref{lemma:1,lemma:2} as for the case $\ell=0$. The difference lies in the initialization of algorithm \alg. Namely, denoting by
\begin{equation*}
m_\ell:=
\begin{dcases*}
3\ell&if $0≤\ell≤n-1$,\\ 3n-1&if $\ell=n$,\\ 3n&if $\ell=n+1$,
\end{dcases*}
\end{equation*} 
one must initialize \alg with the triple $(A_{m_l},B_{m_l},C_{m_l})$ to get the $H$-fraction of $\Phi_n^{(\ell)}$.  In other words, \emph{the $H$-fraction of $\Phi_n^{(\ell)}$ when $1≤\ell≤n+1$ is obtained from the one of $\Phi_n^{(0)}=\Phi_n$ by truncating the first $m_\ell$ terms, and by adjusting the power of the monomial in the first numerator:} because of \eqref{Hfracalgo}, it must be set to $k_{m_\ell}$ instead of $k_{m_{\ell-1}}+k_{m_\ell}+2$.

To sum up the discussion, we have proved the following analogue of \cref{MainHF}.

\begin{thm}
\label{HFPhishifts}
Write the $H$-fraction expansion \eqref{Hfrac-Phin} of $\Phi_n$ as
\begin{equation*}
	\Phi_n(q)= \frac{1}{\be_0}\cfp  \CF_{i=1}^{+\infty}\fr{\al_i}{\be_i}.
\end{equation*}
Then, for any $\ell\in\{1,\ldots,n+1\}$, the $H$-fraction expansion of $\Phi_n^{(\ell)}$ is given by the following formulas:
\begin{align*}
	\Phi_n^{(\ell)}(q)&=\frac{1}{\be_{3\ell}}\cfp  \CF_{i=3\ell+1}^{+\infty}\fr{\al_i}{\be_i}
	\qquad \text{for $1≤ \ell ≤ n-1$},\\
	\Phi_n^{(n)}(q)&=\frac{q^n}{\be_{3n-1}}\cfp  \CF_{i=3n}^{+\infty}\fr{\al_i}{\be_i},\\
	\Phi_n^{(n+1)}(q)&=\frac{q^{n-1}}{\be_{3n}}\cfp  \CF_{i=3n+1}^{+\infty}\fr{\al_i}{\be_i}.
\end{align*}
\end{thm}

\begin{ex} Recall that the explicit $H$-fraction of $\Phi_5$ was given in \cref{ex-Phi5} and that its power series expansion reads, by \eqref{SE-Phin}:
\begin{equation*}
\Phi_5(q)=1+q+q^2+q^3+q^{4}+q^{10}+\sum_{i=11}^{+\infty}\kappa_i q^i.
\end{equation*}
By the previous result, we derive the $H$-fraction of $\Phi^{(1)}$ from \eqref{Hfrac-Phi5} by deleting the first three terms and setting the first numerator to $1$, hence:
\begin{align*}
\Phi_5^{(1)}(q)
	= & 1+q+q^2+q^{3}+q^{9}+\sum_{i=10}^{+\infty}\kappa_{i+1} q^{i}\\
	= & \fr{1}{1-q }\cfp \fr{q^4}{[q]_4}\cfp
\fr{q^5}{q^2 + 1} \cfp \fr{q^3}{1-q } \cfp \fr{q^3}{ [q]_3 }
\cfp \fr{q^5}{ [q]_4-q }\cfp
\fr{q^4}{1-q } \cfp \fr{q^2}{  [q]_2} \\
&
\cfp
\fr{q^5}{ [q]_5-q } \cfp
\fr{q^5}{1} \cfp
\fr{-q^6}{\langle 5\rangle_q + q^6   } \cfp
\fr{q^{11}}{ \langle 5\rangle_q}  
\cfp
\fr{q^{11}}{ \langle 5\rangle_q +q^6} \cfp 
\fr{-q^6}{1} \cfp
\fr{q^5}{ [q]_5-q}
\cfp \cfd
\end{align*}
If instead, we delete the first six terms and set the first numerator to $1$ we get:
\begin{align*}
\Phi_5^{(2)}(q)
	=&1+q+q^{2}+q^{8}+\sum_{i=9}^{+\infty}\kappa_{i+2} q^{i}\\
	=&\fr{1}{1-q } \cfp \fr{q^3}{ [q]_3} \cfp \fr{q^5}{  [q]_4 -q}\cfp
\fr{q^4}{1-q} \cfp
\fr{q^2}{  [q]_2} \cfp
\fr{q^5}{ [q]_5-q  } \cfp\fr{q^5}{1}\\
&  \cfp
\fr{-q^6}{   \langle 5\rangle_q +q^6 } \cfp
\fr{q^{11}}{ \langle 5\rangle_q} 
 \cfp\fr{q^{11}}{ \langle 5\rangle_q +q^6}
\cfp \fr{-q^6}{1} \cfp
	\fr{q^5}{ [q]_5-q }
\cfp \cfd
\end{align*}
A last example: deleting the first $14$ terms in \eqref{Hfrac-Phi5} and setting first numerator to $q^5$ we obtain
\begin{align*}
\Phi_5^{(5)}(q)
&=q^{5}+\sum_{i=6}^{+\infty}\kappa_{i+5} q^{i}
=\fr{q^{5}}{\langle 5\rangle_q } \cfp
\fr{q^{11}}{ \langle 5\rangle_q +q^6 } \cfp \fr{-q^6}{1}\
\cfp
\fr{q^5}{ [q]_5-q}
\cfp \cfd
\end{align*}
\end{ex}

\subsection{The contiguity relations}

Let us proceed with a reformulation of \cref{HFPhishifts}. Recall that the $H$-fraction components $U_n,V_n,W_n$ of $\Phi_n$ were defined in \cref{MainHF} and set
\begin{equation*}
U_n^{(\ell)}(q)=
\begin{dcases*}
\fr{1}{1-q}\cfp\CF_{i=\ell}^{n-3}\left(\fr{q^{n-i}}{[n-i]_q}\cfp \fr{q^n}{[i+2]_q-q}\cfp \fr{q^{i+2}}{1-q}\right)\\[-4pt]
\qquad\qquad\cfp\fr{q^2}{[2]_q}\cfp \fr{q^n}{[n]_q-q}\cfp \fr{q^n}{1},
&if $0≤\ell≤n-3$;\\
\fr{1}{1-q}\cfp\fr{q^2}{[2]_q}\cfp \fr{q^n}{[n]_q-q}\cfp \fr{q^n}{1},&if $\ell=n-2$;\\
\fr{1}{q},&if $\ell=n-1$;\\
0,&if $\ell=n$ or $\ell=n+1$,
\end{dcases*}
\end{equation*}
and
\begin{equation*}
V_n^{(\ell)}(q)=
\begin{dcases*}
V_n(q),&if $0≤\ell≤n-1$;\\
\fr{q^{n}}{\cron}\cfp\fr{q^{2n+1}}{\cron+q^{n+1}}\cfp\fr{-q^{n+1}}{1},&if $\ell=n$;\\
\fr{q^{n-1}}{\cron+q^{n+1}}\cfp\fr{-q^{n+1}}{1},&if $\ell=n+1$.
\end{dcases*}
\end{equation*}
Then \cref{HFPhishifts} says exactly that, for any $\ell\in\{0,1,\ldots,n+1\}$, we have:
\begin{equation*}
\Phi_n^{(\ell)}(q)=U_n^{(\ell)}(q)\cfp V_n^{(\ell)}(q)\cfp \left(W_n(q)\cfp U_n(q)\cfp W_n(q)\cfp\right)^*.
\end{equation*}
From this expression we derive immediately the following consequence.

\begin{cor}\label{contiguityPhin}\emph{Contiguity relations for $\Phi_n^{(\ell)}$.}\\
If $\ell\in\{0,1,\ldots,n-2\}$ then
\begin{equation*}
\Phi_n^{(\ell)}(q)=\fr{1}{1-q}\cfp \fr{q^{n-\ell}}{[n-\ell]_q}\cfp \fr{q^n}{[\ell+2]_q-q}\cfp q^{\ell+2}\,\Phi_n^{(\ell+1)}(q).
\end{equation*}
For the two remaining cases we have
\begin{equation*}
\Phi_n^{(n-1)}(q)=\fr{1}{1}\cfp \fr{-q^{n+1}}{\cron+q^{n+1}}\cfp q^{n+1}\,\Phi_n^{(n)}(q).
\end{equation*}
and
\begin{equation*}
\Phi_n^{(n)}(q)=\fr{q^{n}}{\cron+ q^{n+2}\,\Phi_n^{(n+1)}(q)}.
\end{equation*}
\end{cor}

We now read the effect of these relations on Hankel determinants. 
\begin{proof}[Proof of \cref{MainContiguity}] We treat separately three cases.
\begin{itemize}[wide]
\item Assume first that $\ell\in\{0,1,\ldots,n-2\}$, so that we have, by the previous corollary:
\begin{equation}
\Phi_n^{(\ell)}(q)=\fr{1}{1-q-q^2\,G_1(q)}\label{aaa1}
\end{equation}
where
\begin{equation}
G_1(q)=\fr{-q^{n-\ell-2}}{1+q+\cdots+q^{n-\ell-1}-q^{n-\ell}\,G_2(q)}\label{aaa2}
\end{equation}
and
\begin{equation}
G_2(q)=\fr{-q^{\ell}}{1+q^2+q^3+\cdots+q^{\ell+1}-q^{\ell+2}\,\bigl(-\Phi_n^{(\ell+1)}(q)\bigr)}.\label{aaa3}
\end{equation}
Now we point out the following result.

\begin{lem}[Lemma 2.2 of~\cite{Han16}]
\label{lem22}
Let $k$ be a nonnegative integer and let $F(q), G(q)$ be two power series such that
$$
F(q)=\frac{q^k}{1+q\,u(q) - q^{k+2}\,G(q)}
$$
where $u(q)$ is a polynomial such that $\deg(u)≤k$. Then,
$\De_j(G ) = (-1)^{\frac{k(k+1)}{2}}\,\De_{j+k+1}(F).$
\end{lem}

Applying this lemma successively to \eqref{aaa3}, \eqref{aaa2} and \eqref{aaa1} we get
\begin{align*}
\De_j(\Phi_n^{(\ell+1)})&=(-1)^j\De_j(-\Phi_n^{(\ell+1)})\\
&=(-1)^j(-1)^{\frac{\ell(\ell+1)}{2}}\De_{j+\ell+1}(-G_2)\\
&=(-1)^j(-1)^{\frac{\ell(\ell+1)}{2}}(-1)^{j+\ell+1}\De_{j+\ell+1}(G_2),
\end{align*}
as well as
\begin{align*}
\De_{j+\ell+1}(G_2)&=(-1)^{\frac{(n-\ell-2)(n-\ell-1)}{2}}\De_{j+n}(-G_1)\\
&=(-1)^{\frac{(n-\ell-2)(n-\ell-1)}{2}}(-1)^{j+n}\De_{j+n}(G_1)
\end{align*}
and
\begin{equation*}
\De_{j+n}(G_1)=\De_{j+n+1}(\Phi_n^{(\ell)}).
\end{equation*}
Gathering these formulas we obtain 
\begin{equation*}
\De_j(\Phi_n^{(\ell+1)})=(-1)^{f(n,\ell,j)}\De_{j+n+1}(\Phi_n^{(\ell)})
\end{equation*}
where
\begin{align*}
f(n,\ell,j)&=j+\frac{\ell(\ell+1)}{2}+j+\ell+1+\frac{(n-\ell-2)(n-\ell-1)}{2}+j+n\\
&=\frac{\ell(\ell+3)}{2}+1+\frac{(n-\ell-2)(n-\ell-1)}{2}+3j+n.
\end{align*}
But
\begin{align*}
(n-\ell-2)(n-\ell-1)&=n(n-\ell-2)-\ell(n-\ell-2)-(n-\ell-2)\\
&=n(n-2\ell-3)+\ell(\ell+3)+2,
\end{align*}
hence
\begin{align*}
f(n,\ell,j)&=\frac{\ell(\ell+3)}{2}+1+\fr{n(n-2\ell-3)}{2}+\fr{\ell(\ell+3)}{2}+1+3j+n\\
&\equiv\fr{n(n-2\ell-3)}{2}+j+n\mod{2}\\
&\equiv\fr{n(n+2\ell-1)}{2}+j\mod{2},
\end{align*}
and this last formula implies the contiguity relation \eqref{contiguity} when $\ell=0,1,\ldots,n-2$.

\item Assume now that $\ell=n-1$. By \cref{contiguityPhin} and \eqref{defnq} we have
\begin{equation*}
\Phi_n^{(n-1)}(q)=\fr{1}{1-q^2 G_3(q)}
\end{equation*}
with
\begin{equation*}
G_3(q)= \fr{q^{n-1}}{1+q^2+q^3+\cdots+q^{n-1}+2q^n- q^{n+1}\,\Phi_n^{(n)}(q)}.
\end{equation*}
Applying \cref{lem22} we get
\begin{align*}
\De_j(\Phi_n^{(\ell)})&=(-1)^j \De_j(-\Phi_n^{(\ell)})\\
&=(-1)^j (-1)^{\fr{n(n-1)}{2}}\De_{j+n}(G_1)\\
&=(-1)^j (-1)^{\fr{n(n-1)}{2}}\De_{j+n+1}(\Phi_n^{(n-1)}).
\end{align*}
It is easily seen that this is exactly \eqref{contiguity} for $\ell=n-1$.

\item Lastly, assume  that $\ell=n$. By \cref{contiguityPhin} and \eqref{defnq} we have
\begin{equation*}
\Phi_n^{(n)}(q)=\fr{q^n}{1+q^2+q^3+\cdots+q^{n-1}+2q^n- q^{n+1}-q^{n+2}\,\bigl(-\Phi_n^{(n+1)}(q)\bigr)}.
\end{equation*}
Again, applying \cref{lem22} yields 
\begin{equation*}
\De_j(\Phi_{n}^{(n+1)})=(-1)^j \De_j(-\Phi_{n}^{(n+1)})\\
=(-1)^j (-1)^{\fr{n(n+1)}{2}}\De_{j+n+1}(\Phi_n^{(n)})
\end{equation*}
which is the same as \eqref{contiguity} when $\ell=n$.
\end{itemize}
This finishes our proof of \cref{MainContiguity} in all cases.
\end{proof}

\begin{ex} To illustrate the result of \cref{MainContiguity}, let us consider as usual the case $n=5$. We have, for all $j≥0$:
\begin{align*}
\De_j^{(1)}&=(-1)^j\De_{j+6},&
\De_j^{(2)}&=(-1)^{j+1}\De_{j+6}^{(1)}=-\De_{j+12},\\
\De_j^{(3)}&=(-1)^{j}\De_{j+6}^{(2)}=(-1)^{j+1}\De_{j+18},&
\De_j^{(4)}&=(-1)^{j+1}\De_{j+6}^{(3)}=\De_{j+24},\\
\De_j^{(5)}&=(-1)^{j}\De_{j+6}^{(4)}=(-1)^j\De_{j+30},&
\De_j^{(6)}&=(-1)^{j+1}\De_{j+6}^{(5)}=-\De_{j+36}.\\
\end{align*}
We see that all sequences $\De^{(\ell)}$ are  expressed in terms of the first (non shifted) one $\De=\De^{(0)}$ (remind \cref{fig:Phi5}), using a translation of $n+1=6$ units and a function which changes possibly the signs. \Cref{fig:Phi5-1,fig:Phi5-2,fig:Phi5-3} highlight in particular the translation of the center of the symmetry (the blank circle) when $\ell=1,2,3$.
\end{ex}

\begin{figure}[ht]
  \includegraphics[width=\linewidth]{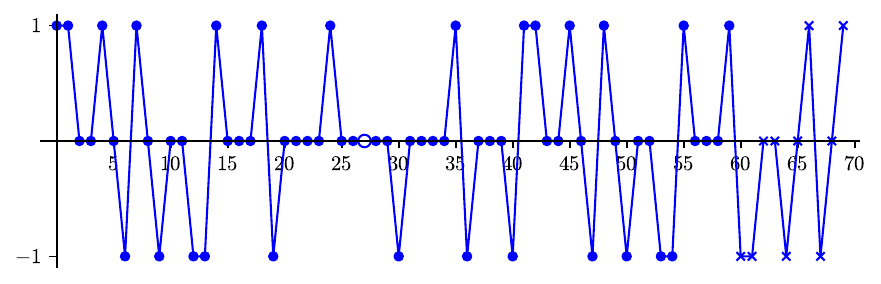}
  \caption{First $70$  determinants in Hankel sequence $\De^{(1)}(\Phi_5)$.}
  \label{fig:Phi5-1}
\end{figure}

\begin{figure}
  \includegraphics[width=\linewidth]{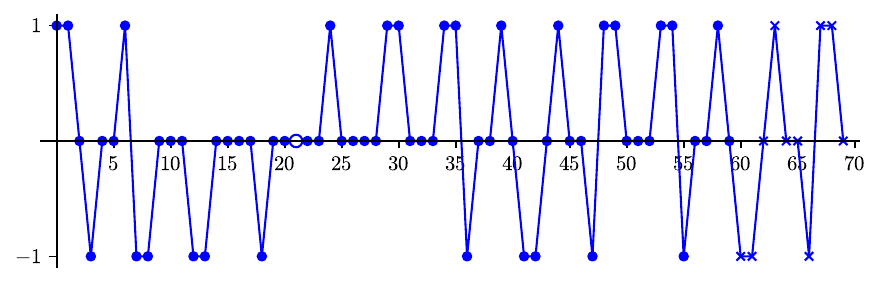}
  \caption{First $70$  determinants in Hankel sequence $\De^{(2)}(\Phi_5)$.}
  \label{fig:Phi5-2}
\end{figure}

\begin{figure}
  \includegraphics[width=\linewidth]{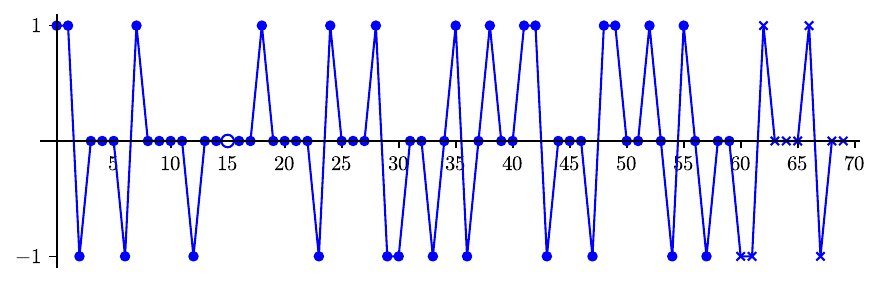}
  \caption{First $70$  determinants in Hankel sequence $\De^{(3)}(\Phi_5)$.}
  \label{fig:Phi5-3}
\end{figure}

\subsection{The shifted case, modulo $p$}

We  end our article by proving \cref{MainModulop}. Fix a prime number $p$ and let $\ell≥0$. Since $\Phi_n^{(\ell)}(q)$ is a power series with integer coefficients, according to Theorem~1.1 of \cite{Han16}, the statement of \cref{MainModulop} reduces to prove that  $\Phi_n^{(\ell)}$ satisfies a quadratic functional equation
\begin{equation*}
A^{(\ell)} +B^{(\ell)} \Phi_n^{(\ell)} + C^{(\ell)} \bigl(\Phi_n^{(\ell)}\bigr)^2 =0
\end{equation*}
with three polynomials $A^{(\ell)},B^{(\ell)},C^{(\ell)}\in\Z[X]$ such that
\begin{equation*}
B^{(\ell)}(0)=1,\quad C^{(\ell)} ≠0, \quad C^{(\ell)} (0)=0.
\end{equation*}
Let us actually prove that such an equation holds with
\begin{equation}\label{eq:condpol}
B^{(\ell)}(0)=1,\quad C^{(\ell)}=q^{\ell+1}.
\end{equation}
We proceed by induction on $\ell≥0$. 
By \eqref{FE-Phin}, it is clear that conditions \eqref{eq:condpol} are fulfilled for $\ell=0$. Now, assume that they are true for some $\ell≥0$. By \cref{lemma:cut} we know that $\Phi_n^{(\ell+1)}$ satisfies the quadratic equation
\begin{equation*}
A^{(\ell+1)} +B^{(\ell+1)} \Phi_n^{(\ell+1)} + C^{(\ell+1)} \bigl(\Phi_n^{(\ell+1)}\bigr)^2 =0
\end{equation*}
with $A^{(\ell+1)},B^{(\ell+1)},C^{(\ell+1)}\in\Z[X]$ satisfying 
\begin{equation*}
B^{(\ell+1)}=B^{(\ell)} + 2\Phi_n^{(\ell)}(0) C^{(\ell)},\quad C^{(\ell+1)}=q C^{(\ell)}.
\end{equation*} 
By induction hypothesis, this gives 
\begin{equation*}
C^{(\ell+1)}=q^{\ell+2},\quad
B^{(\ell+1)}(0)=B^{(\ell)}(0) + 2\Phi_n^{(\ell)}(0) C^{(\ell)}(0)=1,
\end{equation*}
as required.


\newcommand{\etalchar}[1]{$^{#1}$}
\providecommand{\bysame}{\leavevmode\hbox to3em{\hrulefill}\thinspace}

\end{document}